\documentclass[12pt]{amsart}

\usepackage[colorlinks = true,
            linkcolor = red,
            urlcolor  = blue,
            citecolor = blue,
            anchorcolor = blue]{hyperref}

\numberwithin{equation}{section}

\newcommand{\lab}{\label}

\newcommand{\ben}{\begin{enumerate}}
\newcommand{\een}{\end{enumerate}}

\newcommand{\bea}{\begin{eqnarray}}
\newcommand{\ba}{\begin{array}}
\newcommand{\bean}{\begin{eqnarray*}}
\newcommand{\ea}{\end{array}}
\newcommand{\eea}{\end{eqnarray}}
\newcommand{\eean}{\end{eqnarray*}}
\newcommand{\beq}{\begin{equation}}
\newcommand{\eeq}{\end{equation}}
\newcommand{\bthm}{\begin{thm}}
\newcommand{\ethm}{\end{thm}}
\newcommand{\blem}{\begin{lem}}
\newcommand{\elem}{\end{lem}}
\newcommand{\bprop}{\begin{prop}}
\newcommand{\eprop}{\end{prop}}
\newcommand{\bcor}{\begin{cor}}
\newcommand{\ecor}{\end{cor}}
\newcommand{\bdfn}{\begin{dfn}}
\newcommand{\edfn}{\end{dfn}}
\newcommand{\brem}{\begin{rem}}
\newcommand{\erem}{\end{rem}}
\newcommand{\bpf}{\begin{proof}}
\newcommand{\epf}{\end{proof}}
\newcommand{\bfact}{\begin{fact}}
\newcommand{\efact}{\end{fact}}
\newcommand{\bobs}{\begin{obs}}
\newcommand{\eobs}{\end{obs}}

\alph{enumii} \roman{enumiii}

\newtheorem{thm}{Theorem}[section]
\newtheorem{prop}[thm]{Proposition}
\newtheorem{lem}[thm]{Lemma}

\newtheorem{cor}[thm]{Corollary}
\newtheorem{dfn}[thm]{Definition}
\newtheorem{rem}[thm]{Remark}
\newtheorem{fact}[thm]{Fact}
\newtheorem{obs}[thm]{Observation}

             \def\cB{\mathcal B}       
             \def\cF{\mathcal F}
                   
             \def\cV{\mathcal V}       
\def\cS{\mathcal S}
\def\cW{\mathcal W}

\def\endpf{\qed}

\def\N{{\mathbb N}}                \def\Z{{\mathbb Z}}      \def\R{{\mathbb R}}
\def\C{{\mathbb C}}                  
\def\Q{{\mathbb Q}}
\def\1{1\!\!1}
\def\and{\text{ and }}

        \def\diam{\text{\rm {diam}}}

      \def\lra{\longrightarrow}

\def\h{\rm{h}}
\def\hmu{\h_\mu}           
\def\H{\text{{\rm H}}}     \def\HD{\text{{\rm HD}}}   
\def\BD{\text{{\rm BD}}}         \def\PD{\text{PD}}
    
\def\Int{\text{{\rm Int}}}
         \def\P{\text{{\rm P}}}

                   \def\Pa{{\mathcal P}}

\def\a{\alpha}                \def\b{\beta}             \def\d{\delta}
               \def\e{\varepsilon}          
\def\g{\gamma}                \def\Ga{\Gamma}           \def\l{\lambda}
\def\La{\Lambda}              \def\om{\omega}           
               \def\sg{\sigma}
\def\Th{\Theta}               \def\th{\theta}           
\def\ka{\kappa}

\def\bi{\bigcap}              \def\bu{\bigcup}
\def\({\bigl(}                \def\){\bigr)}
\def\lt{\left}                \def\rt{\right}

\def\ld{\ldots}                        \def\^{\tilde}

\def\es{\emptyset}            \def\sms{\setminus}
\def\sbt{\subset}             \def\spt{\supset}
\def\gek{\succeq}             \def\lek{\preceq}
     
      \def\imp{\Rightarrow}
\def\comp{\asymp}
           \def\downto{\searrow}
\def\sp{\medskip}             \def\fr{\noindent}        
\def\ov{\overline}            

\def\fr{\noindent}

\def\om{\omega}

\def\endpf{{\hfill $\square$}}

\newcommand{\ep}{\varepsilon}

\newcommand{\ga}{\gamma}

\begin{document}
\title[]
{ \bf\large {\Large S}kew product Smale endomorphisms over countable shifts of finite type}
\date{\today}
\author[\sc Eugen MIHAILESCU]{\sc Eugen Mihailescu}
\address{Eugen Mihailescu,
Institute of Mathematics of the Romanian Academy,
P.O Box 1-764,
RO 014700, Bucharest, Romania}
\email{Eugen.Mihailescu@imar.ro  \  \ \hspace*{0.5cm}
Web: www.imar.ro/$\sim$mihailes}
\author[\sc Mariusz URBA\'NSKI]{\sc Mariusz URBA\'NSKI}
\address{Mariusz Urba\'nski, Department of Mathematics,
 University of North Texas, Denton, TX 76203-1430, USA}
\email{urbanski@unt.edu \ \  \hspace*{0.5cm} Web: www.math.unt.edu/$\sim$urbanski}
%
%
\date{}
\thanks{Research of the first author supported in part by project PN III-P4-ID-PCE-2016-0823 Dynamics and Differentiable Ergodic Theory,  from UEFISCDI Romania. Research of the second author supported in part by the NSF Grant DMS 0400481.}

\begin{abstract}
We introduce and study skew product Smale endomorphisms over finitely irreducible  shifts with countable alphabets. This case is different from the one with finite alphabets and we develop new methods. In the conformal context we prove that almost all conditional measures of equilibrium states of summable H\"older continuous potentials are exact dimensional, and their dimension is equal to the ratio of (global) entropy and Lyapunov exponent. We show that the exact dimensionality of conditional measures on fibers implies global exact dimensionality of the original measure.  We then study equilibrium states for skew products over expanding Markov-R\'enyi transformations, and settle the question of exact dimensionality of such measures. We apply our results to skew products over the continued fractions transformation. This allows us to extend and improve the Doeblin-Lenstra Conjecture on Diophantine approximation coefficients, to a larger class of measures and irrational numbers. 
\end{abstract}
\maketitle

\textbf{Kewords:} \  Skew product endomorphisms, equilibrium measures, exact dimensional measures, dimensions, stable sets, entropy, natural extensions, continued fractions.

\tableofcontents
\section{Introduction} We introduce and explore skew product Smale endomorphisms modeled on  subshifts of finite type with countable alphabets.  We study the thermodynamic formalism for skew product Smale endomorphisms over countable-to-1 maps, in particular natural extensions of countable--to--1 endomorphisms, such as  EMR (expanding Markov-R\'enyi maps), Gauss map, etc. The case of countable alphabet is different from the one with finite alphabet, since limit sets are usually \textit{non-compact} and new methods are needed to prove exact dimensionality of measures. Also our notion of Smale space is different, although inspired, by the respective notion from \cite{Ru}.

One of our objectives is to develop the thermodynamic formalism of such dynamical systems. In order to do this, we first recall in Section \ref{1sidedshift} the foundations of thermodynamics formalism of one-sided subshifts of finite type modeled on a countable (either finite or infinite) alphabet, from \cite{gdms}, \cite{MU_Israel}. Passing on to the two--sided shifts in Section~ \ref{2sidedshift}, we  provide a  thermodynamic formalism of  
H\"older continuous summable potentials with respect to two--sided subshifts of finite type. This also includes a characterization of Gibbs states in terms of conditional measures which has no counterpart for one sided shifts.

We then define in Section \ref{con} skew product Smale endomorphisms, modeled on countable alphabet subshifts of finite type, and specify its several significant subclasses. 
We define them in Section \ref{CSPSE} and study them in Section \ref{VL}. In Theorem \ref{t1vl4} we show that projections of a.e conditional measures of equilibrium states of summable 
H\"older continuous potentials are exact dimensional, and their dimension is the ratio of global entropy and Lyapunov exponent.  In the proof we develop new methods suited for the countable alphabet case. 
\newline
We prove in Theorem \ref{t1bf3} a version of Bowen's formula giving the Hausdorff dimension of each fiber as the zero of a pressure function; we deal also with the case when pressure function has no zero. Exact dimensional measures have a long history being studied in many settings, for eg \cite{BPS}, \cite{FH}, \cite{Pe}, \cite{PoW}, \cite{Y}, \cite{FM}, \cite{gdms},  \cite{M-stable}, \cite{MU-Adv}, \cite{PU},  to name a few.

Then, in Section \ref{generalskew} we pass to general skew products over \textit{countable-to-1 endomorphisms}. For
endomorphisms, the study of Hausdorff dimension is in general different than for invertible
systems and specific phenomena appear (for eg \cite{Ru2}, \cite{M-stable}, \cite{M-MZ}, \cite{FM}, \cite{M-Mon}, \cite{M-DCDS}, \cite{M-DCDS12}, \cite{MS}, \cite{MU-JFA}). We prove, under
a condition of $\mu$-injectivity for the coding of the base map, the exact dimensionality of
conditional measures of equilibrium measures in stable manifolds fibers, building on \cite{M-stable}.
 
 \fr We  consider general skew product endomorphisms 
$
F: X \times Y \lra X \times Y
$
of the form
$$
F(x, y) = (f(x), g(x, y)),
$$
over countable--to--1 endomorphisms $f: X \to X$, where $X$ is a metric space (not only $E_A^+$), and $Y \subset \R^d$. Then $f$ is coded by a symbol space with countably many symbols, and we find in Theorem~\ref{codedskews} a closed formula for  pointwise dimensions of conditional measures on fibers of $F$.
Then, in Theorem \ref{globalexact} we prove that, if the conditional measures of an equilibrium measures $\mu_\phi$ on fibers are exact dimensional, and if the projection of $\mu_\phi$ in the base is also exact dimensional, then  $\mu_\phi$ is exact dimensional \textit{globally}.

\fr Next, we study several main classes of skew product endomorphisms over countable-to-1 maps, in particular natural extensions (inverse limits); for dynamics on inverse limits see for eg \cite{Ru2}, \cite{M-DCDS12}, \cite{M-MZ}, \cite{M-stable}, \cite{M-Mon},  \cite{M-JSP11}, \cite{M-Cam}. \
In Section \ref{S-EMR} we study EMR (expanding Markov-R\'enyi) maps $f:I \to I$  and conformal Smale skew product  endomorphisms $F:I \times Y \to I\times Y$ over $f$. In Theorem \ref{EMRskews} we prove exact dimensionality of conditional measures on fibers of $F$. In particular, we consider the continued fraction Gauss map
$$
G(x) = \lt\{\frac 1x \rt\},  \   x \in (0, 1],
$$
coded by a countable alphabet; and the Manneville-Pomeau maps 
$$
f_\a(x)= x +x^{1+\alpha} \ (\text{mod} \ 1),  \  \  x\in [0, 1],
$$
$\alpha >0$. In Theorem \ref{f1exact} we provide a class of  equilibrium measures are exact dimensional globally on $I \times Y$.

In Section \ref{doeblin}, we apply our results to Diophantine approximation of irrational numbers $x$, and we generalize the \textit{Doeblin-Lenstra conjecture} about the approximation coefficients $\Theta_n(x)$ in continued fractions representation,  to equilibrium measures $\mu_s$ of potentials 
$$
-s\log|G'|:[0,1]\lra\R, \  \  s>1/2,
$$ 
where we recall that $G$ is the Gauss map $G(x) = \lt\{\frac 1x \rt\},  \   x \in (0, 1]$. If the continued fraction representation of an irrational number $x\in [0, 1)$ is 
$$
x = \frac{1}{a_1+\frac{1}{a_2+\frac{1}{a_3+\ldots}}} = [a_1, a_2, \ldots],
$$
with $a_i\ge 1$ being integer for all $i \ge 1$, and if 
$
\frac{p_n(x)}{q_n(x)} = [a_1, \ldots, a_n] \in \mathbb Q, n \ge 1,
$
then the approximation coefficients (see \cite{IK}) are: 
$$
\Theta_n(x) := q_n(x)^2\cdot \lt|x- \frac{p_n(x)}{q_n(x)}\rt|, \  \  n \ge 1
$$ 
\fr The original Doeblin--Lenstra Conjecture (for eg \cite{BJW}, \cite{IK}) gives information about the frequency of having consecutive pairs ($\Theta_k(x), \Theta_{k-1}(x)$) in some prescribed set, and it involves the lift of the Gauss measure $\mu_G$ to the natural extension space $[0, 1)^2$ of the continued fraction transformation; thus, it is valid for Lebesgue--a.e $x \in [0, 1)$. By contrast, in our case we take the numbers $x$ from the complement of this set. The natural extension $([0, 1)^2, \^G)$ of $G$ is  a skew product which falls into our setting. Hence, we can apply in Theorem~\ref{distr} the results obtained above. Using the exact dimensionality of  $\hat \mu_s$  on the natural extension, we also make the Doeblin--Lenstra Conjecture more precise. Namely in Theorem \ref{dioph}, we \textit{estimate the asymptotic frequency} of 
having the pairs $(\Theta_k(x), \Theta_{k-1}(x))$ $r$-close to pairs $(z, z')$, for $1 \le k \le n$ if $n$ becomes larger and larger, for all irrational numbers $x$ from a set measrable set $\Lambda_s \subset [0, 1)$ with $\mu_s(\Lambda_s)=1$ and $\HD(\Lambda_s) >0$. We emphasize that the Lebesgue measure of $\Lambda_s$ is zero, so it is not covered by the original Doeblin-Lenstra Conjecture. We provide its exact value in formula (2) of Theorem~\ref{dioph}.

\sp Our Smale skew product endomorphisms are related also  to the notion of \textit{chains with complete connections}, introduced in \cite{OM2} and studied for instance in \cite{OM2}, \cite{DF}, \cite{ITM}. Our endomorphisms are related also to \textit{iterated function systems with place-dependent probabilities} (see \cite{BDEG}), which are systems of contractions $\mathcal S = \{\phi_i:X \to X\}_{i \in I}$ with limit set $J_\cS$, where instead of the classical self--similar measure associated to a probability vector $(p_i, i \in I)$, one considers an invariant measure $\mu$ on $J_\cS$ associated to a probabilistic vector composed of variable weights $p_i:X\to [0,1]$, for $i \in I$. 

Many authors studied various related aspects of thermodynamic formalism and its relations to dimension theory, for eg \cite{BPS}, \cite{Bowen_Anosov}, \cite{Bo1}, \cite{Pe}, \cite{FM}, \cite{HMU}, \cite{MM} \cite{gdms}, \cite{MaU-ifs}, \cite{M-MZ}, \cite{M-DCDS}, \cite{M-JSP11}, \cite{M-DCDS12}, \cite{M-ETDS11}, \cite{M-Mon}, \cite{MS}, \cite{MU-Adv}, \cite{M-Cam},  \cite{MSU}, \cite{MU1}, \cite{MU2}, \cite{PS}, \cite{PoW}, \cite{PUZ I}, \cite{PUZ II}, \cite{PU}, \cite{RoUr}, \cite{Ru},  \cite{Sa}, \cite{Sch}, \cite{UZ1}, \cite{UZ2}, and this list is far from complete.

\section{One-Sided Thermodynamic Formalism}\label{1sidedshift}

\fr In this section we collect some fundamental ergodic (thermodynamic 
formalism) results concerning one-sided symbolic dynamics. All of 
them can be found with proofs in \cite{gdms}, \cite{MU_Israel}. Let $E$ be a countable set and let 
$A:E\times E\to\{0,1\}$ be a matrix. A finite or countable infinite 
tuple $\om$ of elements of $E$ is called $A$-admissible if and only 
if $A_{ab}=1$ for any two consecutive elements $a,b$ of $E$. 

The 
matrix $A$ is said to be \textit{finitely irreducible} if there exists a finite
set $F$ of finite $A$-admissible words so that for any two  
elements $a,b$ of $E$ there exists $\g \in F$ such that
the word $a\g b$ is $A$-admissible. In the sequel, 
the incidence matrix $A$ \textit{is assumed to be finitely irreducible}. \
Given $\b>0$, define the metric $d_\b$ on $E^\N$ by 
$$
d_\b\((\om_n)_0^{\infty},(\tau_n)_0^{\infty}\)
=\exp\(-\b\max\{n\ge 0:(0\le k\le n) \imp \om_k=\tau_k\}\)
$$
with the standard convention that $e^{-\infty}=0$. Note that all the
metrics $d_\b$, $\b>0$, on $E^\N$ are H\"older continuously equivalent
and they induce the product topology on $E^\N$. Let
$$
E_A^+=\{(\om_n)_0^{\infty}:\forall_{n\in\N}\ A_{\om_n\om_{n+1}}=1\}
$$
 $E_A^+$ is a closed subset of $E^\N$ and we endow it with
the topology and metrics $d_\b$ inherited from $E^\N$. The 
shift map $\sg:E^\Z\to E^Z$ is defined by the formula
$
\sg\((\om_n)_0^{\infty}\)=\((\om_{n+1})_{n=0}^{\infty}\),
$ and 
 $\sg(E_A^+)\sbt E_A^+$ and $\sg:E_A^+\to E_A^+$ is 
continuous. For every finite word $\om=\om_0\om_1\ld\om_{n-1}$, put $|\om|=n$ 
the length of $\om$, and 
$
[\om]=\{\tau\in E_A^+:\forall_{(0\le j\le n-1)}:\tau_j=\om_j\}$ is the \textit{cylinder} generated by $\om$. Let $\psi:E_A^+\to\R$ continuous, then
 the topological pressure $\P(\psi)$ is 
$$
\P(\psi)
=\lim_{n\to\infty}\frac1n\log\sum_{|\om|=n}\exp\(\sup\(S_n\psi|_{[\om]}\)\)
$$
and the limit exists, as the sequence 
$
\log\sum_{|\om|=n}\exp\(\sup\(S_n\psi|_{[\om]}\), n \in \mathbb N,
$
is sub-additive. The following theorem, a weaker version of the Variational Principle, was proved in \cite{gdms}.

\bthm\lab{2.1.5}
If $\psi:E_A^+\to\R$ is a continuous function and $\mu$ is a $\sg$-invariant Borel 
probability measure on $E_A^+$ such that $\int\psi\, d\mu>-\infty$, then 
$
\h_{\mu}(\sg)+\int_{E_A^+}\psi\, d\mu\le \P(\psi).
$
\ethm

\fr We say that the function $\psi:E_A^+\to\R$ is 
\textit{summable} if and only if
$$
\sum_{e\in E}\exp\(\sup\(\psi|_{[e]}\)\)<+\infty.
$$
A shift-invariant
Borel probability measure $\mu$ on $E_A^+$ is called a Gibbs state
of $\psi$ provided that there are a constant $C\ge 1$ and $P\in\R$
such that
\beq\label{1082305a}
C^{-1}
\le {\mu([\om])\over \exp(S_n\psi(\tau)-Pn)}
\le C
\eeq
for all $n\ge 1$, all admissible words $\om$ of length $n$ and all 
$\tau\in[\om]$. It clearly follows from (\ref{1082305a}) that if $\psi$ admits a Gibbs
state, then $P=\P(\psi)$.

\bdfn\label{1121905}
A function $g:E_A^+\to\C$ is called  H\"older continuous if it is H\"older continuous with respect to one,
equivalently all, metrics $d_\b$. Then $\exists \b>0$ s.t $g$ is Lipschitz continuous with respect to
$d_\b$. The corresponding Lipschitz constant is $L_\b(g)$.
\edfn

\fr The proofs of the following three results can be found for instance in 
\cite{gdms} and \cite{MU_Israel}.  Regarding the result of Theorem \ref{gibbs1s}, it was shown in \cite{Sa1} that if the shift is topologically mixing, then finite irreducibility is also necessary. For Theorem \ref{t1_2015_11_17} see \cite{Sa}.

\bthm\label{gibbs1s}
For every H\"older continuous summable potential $\psi:E_A^+\to\R$ 
there exists a unique Gibbs state $\mu_\psi$ on $E_A^+$. The measure
$\mu_\psi$ is ergodic.
\ethm

\bthm\label{t1_2015_11_17}
Suppose $\psi:E_A^+\to\R$ is a H\"older continuous 
potential. Then, denoting by $\P_F(\psi)$ the topological pressure of $\psi|_{F_A^+}$ with respect to the shift map $\sg:F_A^+\to F_A^+$, we have 
$
\P(\psi)=\sup\{\P_F(\psi)\},
$
where the supremum is taken over all finite subsets $F$ of $E$; equivalently over all finite subsets $F$ of $E$ such that the matrix $A|_{F\times F}$ is irreducible.
\ethm

\bthm[Variational Principle for One-Sided Shifts]\label{t1eqs}
Suppose that $\psi:E_A^+\to\R$ is a H\"older continuous summable 
potential. Then
$$
\sup\lt\{\hmu(\sg)+\int_{E_A^+}\psi d\mu, \ \mu\circ\sg^{-1} = \mu 
   \  \text{and}    \int\psi d\mu>-\infty\rt\}
=\P(\psi)
={\hmu}_\psi(\sg)+\int_{E_A^+}\psi d\mu_\psi,
$$
and $\mu_\psi$ is the only measure at which this supremum is attained.
\ethm

\fr Any measure that realizes the supremum in the above Variational
Principle is called an equilibrium state for $\psi$. Then
Theorem~\ref{t1eqs} can be reformulated as follows.

\bthm\label{t3111105p70} 
If $\psi:E_A^+\to\R$ is a H\"older continuous summable potential, then the 
Gibbs state $\mu_\psi$ is a unique equilibrium state for $\psi$. 
\ethm

\fr Also due to the irreducibility of the incidence matrix $A$, we have:

\bprop\label{p1_2015_11_14}
A H\"older continuous $\psi:E_A^+\to\R$ is summable if and only if $\P(\psi)<\infty$. 
\eprop

\section{Two-Sided Thermodynamic Formalism}\label{2sidedshift}

\fr As in the previous section, let $E$ be a countable set and let 
$A:E\times E\to\{0,1\}$ be a
finitely irreducible matrix. Given $\b>0$ we define the metric $d_\b$
on $E^\Z$ by:
$$
d_\b\((\om_n)_{-\infty}^{\infty},(\tau_n)_{-\infty}^{\infty}\)
=\exp\(-\b\max\{n\ge 0:\forall_{k\in\Z}|k|\le n\, \imp\, \om_k=\tau_k\}\)
$$
with the standard convention that $e^{-\infty}=0$. Note that all the
metrics $d_\b$, $\b>0$, on $E^\Z$ are H\"older continuously equivalent
and they induce the product topology on $E^\Z$. We set
$$
E_A=\{(\om_n)_{-\infty}^{\infty}:\forall_{n\in\Z}\ A_{\om_n\om_{n+1}}=1\}.
$$
Obviously $E_A$ is a closed subset of $E^Z$ and we endow it with
the topology and metrics $d_\b$ inherited from $E^\Z$. The two-sided
shift map $\sg:E^\Z\to E^Z$ is defined as 
$
\sg\((\om_n)_{-\infty}^{\infty}\)=\((\om_{n+1})_{-\infty}^{\infty}\).
$
Clearly $\sg(E_A)=E_A$ and $\sg:E_A\to E_A$ is a
homeomorphism.

\bdfn\label{1121905B}
A function $g:E_A\to\C$ is said to be H\"older continuous
provided that it is H\"older continuous with respect to one,
equivalently all, metrics $d_\b$. Then there exists at least one (in
fact an open segment)
parameter $\b>0$ such that $g$ is Lipschitz continuous with respect to
$d_\b$. The corresponding Lipschitz constant is denoted by $L_\b(g)$.
\edfn

\fr For every $\om\in E_A$ and all $-\infty\le m\le n\le\infty$, we
set
$
\om|_m^n=\om_m\om_{m+1}\ld\om_n.
$

\fr Let $E_A^*$ be the set of all $A$-admissible finite words. For
$\tau\in E^*$, $\tau=\tau_m\tau_{m+1}\ld\tau_n$, we set
$$
[\tau]_m^n=\{\om\in E_A:\om|_m^n=\tau\}
$$
and call $[\tau]_m^n$ the cylinder generated by $\tau$ of size from
$m$ to $n$. The family of all cylinders of size from $m$ to $n$ will be
denoted by $C_m^n$. 
If $m=0$ we simply write $[\tau]$ for $[\tau]_m^n$. 

\sp\fr  Let $\psi:E_A\to\R$ be a continuous function. The topological pressure
$\P(\psi)$ is defined  by
\beq\label{1_2015_11_04}
\P(\psi)
:=\lim_{n\to\infty}\frac1n\log\sum_{\om\in C_0^{n-1}}\exp\(\sup\(S_n\psi|_{[\om]}\)\),
\eeq
and the limit exists due to the same subadditivity argument. Similarly as in Theorem \ref{2.1.5}, we immediately obtain:

\bthm\lab{2.1.5+-}
If $\psi:E_A\to\R$ is a continuous function and $\mu$ is a $\sg$-invariant Borel 
probability measure on $E_A$ such that $\int\psi\, d\mu>-\infty$, then 
$$
\h_\mu(\sg)+\int_{E_A}\psi\, d\mu\le \P(\psi).
$$
\ethm

\sp A shift-invariant
Borel probability measure $\mu$ on $E_A$ is called a \textit{Gibbs state}
of $\psi$ if there are a constant $C\ge 1$ and $P\in\R$
such that:
\beq\label{1082305}
C^{-1}
\le {\mu([\om|_0^{n-1}])\over \exp(S_n\psi(\om)-Pn)}
\le C
\eeq
for all $n\ge 1$ and all $\om\in E_A$.  It clearly follows from (\ref{1082305}) that if $\psi$ admits a Gibbs
state, then $P=\P(\psi)$. Two functions $\psi_1$ and $\psi_2$ are 
called \textit{cohomologous} in a class $G$ of real-valued functions defined on
$E_A$ if and only if there exists $u\in G$ such that
$$
\psi_2-\psi_1=u-u\circ\sg.
$$
Any function of the form $u-u\circ\sg$ is called a \textit{coboundary} in $G$.
A function $\psi:E_A\to\R$ is called cohomologous to a constant,
say $b\in\R$ provided that $\psi-b$ is a coboundary. Notice that any
two functions on $E_A$, cohomologous in $C(E_A)$, the class
of all real-valued bounded functions on $E_A$, have the same 
topological pressure and the same set of Gibbs measures. 

\sp A function
$\psi:E_A\to\R$ is called \textit{past-independent} if for
every $\tau\in C_0^{\infty}$ and for all $\om,\rho\in[\tau]$, we have
$\phi(\om)=\phi(\tau)$. 

To apply the previous Section, we need the following:

\blem\label{l1082305}
Every H\"older continuous function $\psi:E_A\to\R$ is cohomologous
to a past-independent H\"older continuous function $\psi^+:E_A\to\R$ 
in the class $\H_B$ of all bounded H\"older continuous functions.
\elem

{\sl Proof.} The proof is essentially the same as in \cite{Bowen_Anosov}, 
Lemma~1.6, page 11. For every letter $e\in E$,  fix some infinite word $\ov e\in 
E_A(-\infty,-1)$ such that $A_{\ov e_{-1}e}=1$. Then for 
$\om\in E_A$ fix as before an infinite word $\ov \omega_0\in E_A(-\infty, -1)$, and put $\ov\om=\ov{\om_0}\om|_0^{\infty}$. Hence $\ov \om$ is the same as $\om$ starting from index 0, but has fixed elements with negative indices. Now set:
$$
u(\om)=\sum_{j=0}^\infty(\psi(\sg^j(\om))-\psi(\sg^j(\ov\om)).
$$
Fix 
$\b>0$ so that $\psi$ is Lipschitz continuous w.r.t  $d_\b$. For $j\ge 0$, $[\sg^j(\om)|_{-j}^{\infty}]
=[\sg^j(\ov\om)|_{-j}^{\infty}]$, so $d_\b(\sg^j(\om),
\sg^j(\ov\om))\le e^{-\b j}$, and thus
\beq\label{2082305}
|\psi(\sg^j(\om))-\psi(\sg^j(\ov\om))|\le L_\b(\psi)e^{-\b j}.
\eeq
$u$ is well-defined 
and continuous. If $d_\b(\om,\tau)=e^{-\b n}$ then 
$[\om|_{-n}^n]=[\tau|_{-n}^n]$, and for $0\le j\le n$,
$$
|\psi(\sg^j(\om))-\psi(\sg^j(\tau))|
\le L_\b(\psi)d_\b(\sg^j(\om),\sg^j(\tau))
\le L_\b(\psi)e^{-\b(n-j)}
$$
and
$$
|\psi(\sg^j(\ov\tau))-\psi(\sg^j(\ov\om))|
\le L_\b(\psi)d_\b(\sg^j(\ov\tau),\sg^j(\ov\om))
\le L_\b(\psi)e^{-\b(n-j)}
$$
Thus using also (\ref{2082305}), we get
$$
|u(\om)-u(\tau)|
\le 2L_\b(\psi)\sum_{j=0}^{E(n/2)}e^{-\b(n-j)} +
     2L_\b(\psi)\sum_{j>E(n/2)}e^{-\b j} \le 4L_\b(\psi)(1-e^{-\b})^{-1}e^{-\b{n\over 2}}.
$$
So $u:E_A\to\R$ is Lipschitz continuous with 
respect to the metric $d_{\b/2}$, and by (\ref{2082305}) it is bounded. 
So $u\in\H_{\b/2}$. Hence $\psi^+=\psi-u+u\circ\sg$ is  Lipschitz 
continuous with respect to  $d_{\b/2}$. It follows that
$\psi^+$ is past-independent. 
\endpf

\sp In the setting of the above lemma, let $\ov\psi^+$ be the 
factorization of $\psi^+$ on $E_A^+$, i.e. 
$$
\psi^+=\ov\psi^+\circ\pi_0,
$$
where $\pi_0: E_A \to E_A^+, \ \pi_0(\omega) = (\omega_0 \omega_1 \ldots), \om \in E_A$. As a consequence of this lemma we get,

\blem\label{l3_2015_11_17}
If $\psi:E_A\to\R$ is a H\"older continuous potential, then
$
\P(\psi)=\P(\ov\psi^+),
$ \
where, we remind, the former pressure is taken with respect to the two-sided shift $\sg:E_A\to E_A$ while the latter one is taken with respect to the one-sided shift $\sg:E_A^+\to E_A^+$.
\elem

Then from this lemma and Theorem~\ref{t1_2015_11_17}, we get 

\bthm\label{t2_2015_11_17}
Suppose that $\psi:E_A\to\R$ is a H\"older continuous 
potential. Then, denoting by $\P_F(\psi)$ the topological pressure of $\psi|_{F_A}$ with respect to $\sg:F_A\to F_A$, we have that
$$
\P(\psi)=\sup\{\P_F(\psi)\},
$$
where the supremum is taken over all finite subsets $F$ of $E$; equivalently over all finite subsets $F$ of $E$ such that the matrix $A|_{F\times F}$ is irreducible.
\ethm

We call the function $\psi:E_A\to\R$ is \textit{summable} if and
only if
$$
\sum_{e\in E}\exp\(\sup\(\psi|_{[e]}\)\)<\infty.
$$
As in the case of one-sided shift, we have the following.

\bprop\label{p2_2015_11_14}
A H\"older continuous $\psi:E_A\to\R$ is summable if and only if $\P(\psi)<\infty$. 
\eprop

\fr From Lemma~\ref{l1082305} (the coboundary
appearing there is bounded), we get the following.

\blem\label{l1082405}
Every H\"older continuous summable function $\psi:E_A\to\R$ is cohomologous
to a past-independent H\"older continuous summable function $\psi^+:E_A\to\R$ 
in the class $\H_B$ of all bounded H\"older continuous functions.
\elem

\bthm\label{t1111105p68}
For every H\"older continuous summable potential $\psi:E_A\to\R$ 
there exists a unique Gibbs state $\mu_\psi$ on $E_A$. The measure
$\mu_\psi$ is ergodic.
\ethm
{\sl Proof.} Let $\psi^+$ be the past-independent H\"older continuous 
summable potential ascribed to $\psi$ in Lemma~\ref{l1082405}.
It 
follows from Theorem~\ref{gibbs1s} that there exists a unique 
measure $\mu_\psi^+$ on $E_A^+$ for which 
 (\ref{1082305}) is satisfied. Also $\mu_\psi^+$ is
ergodic. Since $\mu_\psi^+$ is invariant, the formula
$$
\mu_\psi([\om|_m^n])
=\mu_\psi^+(\sg^m([\om|_m^n]))
=\mu_\psi^+([\eta]_0^{n-m}), 
$$
with $\eta_0 = \om_m, \ldots, \eta_{n-m} = \om_n$, for $|\om| = n-m+1$, gives rise to a shift-invariant measure $\mu_\psi$ on 
$E_A$, for which (\ref{1082305}) holds. Thus $\mu_\psi$ 
is a Gibbs state for $\psi$, and  is 
ergodic. Passing to the uniqueness,
if $\mu$ is a Gibbs state for $\psi$, then from its invariance and (\ref{1082305}), for all 
$n\ge 0$ and all $\om\in E_A$, we have 
$$
C^{-1}
\le {\mu([\om|_{-n}^n])\over \exp\(S_{2n+1}\psi(\sg^{-n}(\om))-\P(\psi)n\)}
\le C.
$$
Any Gibbs states of $\psi$ are equivalent and, as one
of them is ergodic, uniqueness follows.  

\endpf

\bthm[Variational Principle for Two-Sided Shifts]\label{t2111105p69}
Suppose that $\psi:E_A\to\R$ is a H\"older continuous summable 
potential. Then
$$
\sup\lt\{\hmu(\sg)+\int_{E_A}\!\!\!\psi d\mu:\mu\circ\sg^{-1}=\mu 
   \  \text{ and }  \  \int\!\!\psi d\mu>-\infty\rt\}
=\P(\psi)
={\hmu}_\psi(\sg)+\int_{E_A}\!\!\psi d\mu_\psi,
$$
and $\mu_\psi$ is the only measure at which this supremum is taken on.
\ethm
{\sl Proof.} We replace $\psi$ by the past-independent H\"older continuous summable potential
$\psi^+$ resulting from Lemma~\ref{l1082405}. Since
the dynamical system $(\sg,E_A)$, is canonically isomorphic to the  natural extension of $(\sg,E_A^+)$, the map $\mu\mapsto
\mu\circ\pi_0^{-1}$ gives a bijection between the space $M_\sg(E_A)$ of $\sigma$-invariant probabilities on $E_A$,  and the space $M_\sg(E_A^+)$ of $\sigma$-invariant probabilities on $E_A^+$,
which preserves entropies. Since $\P(\psi)=\P(\ov\psi^+)$ by Lemma~\ref{l3_2015_11_17}, and since for every 
$\mu\in M_\sg(E_A)$, \ 
$
\int_{E_A^+}\ov\psi^+d\mu\circ \pi_0^{-1}
=\int_{E_A}\ov\psi^+\circ\pi_0 d\mu
=\int_{E_A}\psi^+d\mu.
$
We are done, due to Theorem~\ref{t1eqs}.
\endpf

\sp\fr Any measure that realizes the supremum value in the above Variational Principle is called an \textit{equilibrium state} for $\psi$. Then Theorem~\ref{t2111105p69} can be reformulated as follows.

\bthm\label{t3111105p70A} 
If $\psi:E_A\to\R$ is a H\"older continuous summable potential, then the 
Gibbs state $\mu_\psi$ is a unique equilibrium state for $\psi$. 
\ethm

 We will need however more characterizations of Gibbs states. Let the partition
$$
\Pa_-=\{[\om|_0^{\infty}]:\om\in E_A\}
     =\{[\om]:\om\in E_A^+\}.
$$
 $\Pa_-$ is a measurable partition of $E_A$ and two
elements $\a,\b\in E_A$ belong to the same element of this 
partition if and only if $\a|_0^{\infty}=\b|_0^{\infty}$. If $\mu$
is a Borel probability measure on $E_A$, we let 
$$
\{\ov\mu^\tau:\tau\in E_A\}
$$ 
be a \textit{canonical system of conditional measures} induced by
partition $\Pa_-$ and measure $\mu$ (see Rokhlin \cite{Ro}). Each $\ov\mu^\tau$ is a Borel probability
measure on $[\tau|_0^{\infty}]$ and we will frequently write $\ov\mu^\om$, $\om\in E_A^+$, to denote the corresponding 
conditional measure on $[\om]$. Denote by 
$$\pi_0: E_A \to E_A^+, \ \pi_0(\tau) = \tau|_0^\infty, \tau \in E_A,$$
the canonical projection to $E_A^+$.
The system $\{\ov\mu^\om:\om\in E_A^+\}$  is determined by the fact that:
$$
\int_{E_A}g\,d\mu=\int_{E_A^+}\int_{[\om]}g\,d\ov\mu^{\om}
  \,d(\mu\circ \pi_0^{-1})(\om)
$$
for every measurable function $g\in L^1(\mu)$ (\cite{Ro}). It is evident from this characterization that if we change such a system on a set of zero $\mu\circ \pi_0^{-1}$-measure, then we also obtain a system of conditional measures. The canonical system of conditional measures induced by $\mu$ is uniquely defined up to a set of zero $\mu\circ \pi_0^{-1}$-measure. We say that a collection
$$
\{\ov\mu^\om:\om\in E_A^+\}
$$
defines a \textit{global system of conditional measures} of $\mu$ if this is indeed a system of conditional measures of $\mu$ and a measure $\ov\mu^\om$ is defined for every $\om\in E_A^+$, rather than only on a set of full $\mu\circ \pi_0^{-1}$-measure. The first characterization of Gibbs states is the following.

\bthm\label{t4111105p70}
Suppose that $\psi:E_A\to\R$ is a H\"older continuous summable 
potential. Let $\mu$ be a Borel probability shift-invariant measure on
$E_A$. Then $\mu=\mu_\psi$, the unique Gibbs state for $\psi$ if
and only if there exists $D\ge 1$ such that 
\beq\label{1111105}
D^{-1}
\le {\ov\mu^\om([\tau\om|_{-n}^{\infty}])\over 
    \exp\(S_n\psi(\rho)-\P(\psi)n\)}
\le D
\eeq
for every $n\ge 1$, $\mu\circ \pi_0^{-1}$-a.e. $\om\in E_A^+$, $\ov\mu^\om$-a.e. $\tau\om\in E_A(-n,\infty)$ with
$A_{\tau_{-1}\om_0}=1$, and $\rho\in
[\tau\om|_0^{\infty}]$. Also there exists a global system of conditional measures of $\mu_\psi$ s.t,
\beq\label{1111105B}
D^{-1}
\le {\ov\mu_\psi^\om([\tau\om|_{-n}^{\infty}])\over 
    \exp\(S_n\psi(\rho)-\P(\psi)n\)}
\le D
\eeq
for every $\om\in E_A^+$,  $n\ge 1$,  $\tau\in E_A(-n,-1)$ with
$A_{\tau_{-1}\om_0}=1$, and every $\rho\in[\tau\om|_0^{\infty}]$. 
\ethm
{\sl Proof.} Suppose (\ref{1111105}) holds. For every $\om\in
E_A$ (note that here indeed "for every'', although \eqref{1111105} is assumed to hold only for $\mu\circ \pi_0^{-1}$-a.e. $\om\in E_A^+$) and every $n\ge 1$, if $\om|_0^{n-1}|_{-n}^{-1}$ denotes the finite word $\eta_{-n}\ldots \eta_{-1}$ with $\eta_{-n} = \om_0, \ldots, \eta_{-1} = \om_{n-1}$, then we obtain:
\beq\label{2111105}
\aligned
\mu([\om|_0^{n-1}])
&=\mu(\sg^n([\om|_0^{n-1}]))
 =\mu([\om|_0^{n-1}|_{-n}^{-1}])
 =\int_{E_A^+}\ov\mu^\tau\([\om|_0^{n-1}|_{-n}^{-1}\tau]\) \ d\mu\circ\pi_0^{-1}(\tau)\\
&=\int_{E_A^+:A_{\om_{n-1}\tau_0}=1}\ov\mu^\tau\([\om|_0^{n-1}|_{-n}^{-1}\tau]\) \ 
       d\mu\circ\pi_0^{-1}(\tau) \\
       &\comp\exp\(S_n\psi(\om)-\P(\psi)n\)\sum_{e\in E:A_{\om_{n-1}e}=1}\mu([e])
\endaligned
\eeq
Consequently,
\beq\label{3111105}
\mu([\om|_0^{n-1}])\lek\exp\(S_n\psi(\om)-\P(\psi)n\).
\eeq
In order to prove the opposite inequality, notice that because of finite
irreducibility of the matrix $A$ there exists a finite set $F\sbt E$ such
that for every $a\in E$ there exists $b\in F$ such that $A_{ab}=1$. Since
$\mu$ is a non-zero measure, there exists $c\in E$ such that $\mu([c])>0$. 
Invoking finite irreducibility of the matrix $A$ again, we see that for every
$e\in E$ there exists a finite word $\a$ such that $e\a c$ is $A$-admissible.
Put $k=|e\a|$. It then follows from (\ref{2111105}) that
$$
\mu([e])
\ge \mu([e\a])
\gek \exp\(S_k\psi(\rho)-\P(\psi)k\)\mu([c])>0
$$
for every $\rho\in [e\a]$. Hence $T=\min\{\mu([\e]):e\in F\}>0$. Continuing
(\ref{2111105}), we see that
$
\mu([\om|_0^{n-1}])\gek T\exp\(S_n\psi(\om)-\P(\psi)n\).
$
Combining this with (\ref{3111105}) we see that $\mu$ is a Gibbs state for $\psi$, and the first assertion of the theorem is established.

\sp Now,  to complete the proof, we need to define a global system of conditional measures of $\mu_\psi$ such that \eqref{1111105B} holds
for every $\om\in E_A^+$,  $n\ge 1$, $\tau\in E_A(-n,-1)$ with
$A_{\tau_{-1}\om_0}=1$, and every $\rho\in\sg^{-n}([\tau\om|_{-n}^{\infty}])=
[\tau\om|_0^{\infty}]$. Indeed, let 
$
L:\ell_\infty\to\ell_\infty
$
be a Banach limit. Note that:
\beq\label{3_2015_11_14}
\begin{aligned}
{\mu_\psi\([\tau\om|_{-n}^{k-1}]\)\over \mu_\psi\(\om|_0^{k-1}]\)}
 &={\mu_\psi\([\tau\om|_0^{n+k-1}]\)\over \mu_\psi\(\om|_0^{k-1}]\)}
 \comp{\exp\(S_{n+k}\psi(\rho)-\P(\psi)(n+k)\)\over
   \exp\(S_k\psi(\sg^n(\rho)-\P(\psi)k \)} 
   =e^{S_n\psi(\rho)-\P(\psi)n} \\
   &\comp\mu_\psi([\tau]_0^{n-1}), 
 \end{aligned}
\eeq
belongs to $\ell_\infty$ (comparability constants from Gibbs property of $\mu_\psi$). So the sequence
$$
\lt({\mu_\psi\([\tau\om|_{-n}^{k-1}]\)\over \mu_\psi\(\om|_0^{k-1}]\)}\rt)_{k=1}^\infty
$$
belongs to $\ell_\infty$. We can then define
$$
\ov\mu_\psi^\om\([\tau\om|_{-n}^{\infty}]\):=L\lt(\lt({\mu_\psi\([\tau\om|_{-n}^{k-1}]\)\over \mu_\psi\(\om|_0^{k-1}]\)}\rt)_{k=1}^\infty\rt).
$$
For every $g:[\om]\to\R$, and a linear combination 
$\sum_{j=1}^sa_j\1_{[\tau^{(j)}\om|_{-n_j}^{\infty}]}$, the sequence
\beq\label{2_2015_11_14}
{\mu_\psi\lt(\sum_{j=1}^sa_j\1_{[\tau^{(j)}\om|_{-n_j}^{k-1}]}\rt)
  \over \mu_\psi\(\om|_0^{k-1}]\)}
\comp\mu_\psi\lt(\sum_{j=1}^sa_j\1_{[\tau^{(j)}]_0^{n_j-1}}\rt),
\eeq
with the same comparability constants as above, belongs to $\ell_\infty$. We can then define
$$
\ov\mu_\psi^\om\lt(\sum_{j=1}^sa_j\1_{[\tau^{(j)}\om|_{-n_j}^{\infty}]}\rt)
:=L\lt(\lt({\mu_\psi\lt(\sum_{j=1}^sa_j\1_{[\tau^{(j)}\om|_{-n_j}^{k-1}]}\rt)
  \over \mu_\psi\(\om|_0^{k-1}]\)}\rt)_{k=1}^\infty\rt).
$$
So, we have defined a function $\ov\mu_\psi^\om$ from the vector space $\cV$ of all linear combinations as above the the set of real numbers. Since the Banach limit is a positive linear operator, so is the function $\ov\mu_\psi^\om:\cV\to\R$. Furthermore, because of monotonicity of Banach limits, and because of \eqref{2_2015_11_14}, $\ov\mu_\psi^\om(g_n)\downto 0$ whenever $(g_n)_{n=1}^\infty$ is a monotone decreasing sequence of functions in $\cV$ converging pointwise to $0$. Therefore, Daniell-Stone Theorem gives a unique Borel probability measure on $[\om]$, whose restriction to $\cV$ coincides with $\ov\mu_\psi^\om$. We keep the same symbol $\ov\mu_\psi^\om$ for this extension. Now, it follows from Martingale Convergence Theorem that, for $\mu_\psi\circ \pi_0^{-1}$-a.e. $\om\in E_A^+$ and every $\tau\in E_A(-n,-1)$ with
$A_{\tau_{-1}\om_0}=1$ the limit 
$$
\lim_{k\to\infty}{\mu_\psi\([\tau\om|_{-n}^{k-1}]\)\over\mu_\psi\(\om|_0^{k-1}]\)}
$$
exists and equals the conditional measure of $\mu_\psi$ on $[\om]$. By properties of Banach limits,
$$
{\mu_\psi\([\tau\om|_{-n}^{k-1}]\)\over \mu_\psi\(\om|_0^{k-1}]\)}=\mathop{\lim}\limits_{k\to\infty}{\mu_\psi\([\tau\om|_{-n}^{k-1}]\)\over\mu_\psi\(\om|_0^{k-1}]\)},
$$
and thus the collection
$
\{\ov\mu_\psi^\om:\om\in E_A^+\}
$
is indeed a global system of conditional measures of $\mu_\psi$. Using also \eqref{3_2015_11_14} this completes the proof.
\endpf

\sp Similarly, let
$$
\Pa_+=\{[\om|_{-\infty}^{-1}]:\om\in E_A\},
$$
and given a Borel probability measure $\mu$ on $E_A$, let 
$\{\mu^{+\om}:\om\in E_A\}$ be the corresponding canonical system
of conditional measures. As in Theorem~\ref{t4111105p70}, we  
prove the following.

\bthm\label{t3111705p148}
Suppose $\psi:E_A\to\R$ is a H\"older continuous summable 
potential. Let $\mu$ be a Borel probability shift-invariant measure on
$E_A$. Then $\mu$ is equal to the unique Gibbs state $\mu_\psi$ of $\psi$, if
and only if there exists $D\ge 1$ s.t for all $\om\in E_A(-\infty,-1)$,  $n\ge 1$, $\tau\in 
E_A(0,n-1)$ with $A_{\om_{-1}\tau_0}=1$, and 
$\rho\in [\om\tau|_{-\infty}^{n-1}]$, the conditional measures $\mu^{+\om}$ satisfy, 
\beq\label{1111105E}
D^{-1}
\le {\mu^{+\om}([\om\tau|_{-\infty}^{n-1}])\over 
    \exp\(S_n\psi(\rho)-\P(\psi)n\)}
\le D
\eeq

\ethm

\

\section{Skew product Smale Spaces of Countable Type}\label{con}

\fr We keep the notation from the previous two sections.  
\begin{dfn}\label{Smaleskew}
Let $(Y,d)$
be a complete bounded metric space, and take for every $\om
\in E_A^+$ an arbitrary set $Y_\om\sbt Y$ and a continuous 
injective map $T_\om:Y_\om\to Y_{\sg(\om)}$. 
Define
$$
\hat Y:=\bu_{\om\in E_A^+}\{\om\}\times Y_\om\sbt E_A^+\times Y.
$$
Define the map $T:\hat Y\to\hat Y$ by \
$
T(\om,y)=(\sg(\om),T_\om(y)). \
$
The pair $(\hat Y,T:\hat Y\to\hat Y)$ is called a \textit{skew product 
Smale endomorphism} if there exists $\l>1$ such 
that  $T$ is fiberwise uniformly contracting, i.e for all $\om\in E_A^+$ and all $y_1,y_2\in Y_\om$, 
\beq\label{1111205}
d(T_\om(y_2),T_\om(y_1))\le \l^{-1}d(y_2,y_1)
\eeq

\end{dfn}

Note that for every $\tau\in E_A(-n,\infty)$ the composition
$
T_\tau^n
=T_{\tau|_{-1}^{\infty}}\circ T_{\tau|_{-2}^{\infty}}\circ\ld
   \circ T_{\tau|_{-n}^{\infty}}:Y_\tau\to Y_{\tau|_0^{\infty}}
$
is well-defined. Therefore for every $\tau\in E_A$ we can define the map
$$
T_\tau^n
:=T_{\tau|_{-n}^{\infty}}^n
:=T_{\tau|_{-1}^{\infty}}\circ T_{\tau|_{-2}^{\infty}}\circ\ld
  \circ T_{\tau|_{-n}^{\infty}}:Y_{\tau|_{-n}^{\infty}}\longrightarrow Y_{\tau|_0^{\infty}}
$$
Then the sequence $\(T_\tau^n\(Y_{\tau|_{-n}^{\infty}}\)\)_{n=0}
^\infty$ consists of descending sets, and
\beq\label{3111205}
\diam\(T_\tau^n\(Y_{\tau|_{-n}^{\infty}}\)\)\le \l^{-n}\diam(Y).
\eeq
The same is then true for the closures of these sets, i.e. we have that
the sequence $\(\ov{T_\tau^n\(Y_{\tau|_{-n}^{\infty}}\)}\)_{n=0}
^\infty$ consists of closed descending sets, and
$
\diam\(\ov{T_\tau^n\(Y_{\tau|_{-n}^{\infty}}\)}\)\le \l^{-n}\diam(Y).
$
Since the metric space $(Y,d)$ is complete, we conclude that 
its intersection
$$
\bi_{n=1}^\infty \ov{T_\tau^n\(Y_{\tau|_{-n}^{\infty}}\)}
$$
is a singleton. Denote its only element by $\hat\pi_2(\tau)$. So, we have defined the map
$$
\hat\pi_2:E_A\to Y,
$$
and next define the map $\hat\pi:E_A\to E_A^+\times Y$ by the formula
\beq\label{5111705p141}
\hat\pi(\tau)=(\tau|_0^{\infty},\hat\pi_2(\tau)),
\eeq
and the truncation to the elements of non-negative indices by $$\pi_0: E_A \to E_A^+, \  \pi_0(\tau) = \tau|_0^\infty$$
In the notation for $\pi_0$ we drop the hat symbol, as this projection is in fact independent of the skew product on $\hat Y$.
For all $\om\in E_A^+$ define the $\hat \pi_2$-projection of the cylinder $[\om]\subset E_A$, 
$$
J_\om:=\hat\pi_2([\om])\in Y,
$$
and call these sets the \textit{stable Smale fibers} of the system $T$. The global invariant set is:
$$
J:=\hat\pi(E_A)=\bu_{\om\in E_A^+}\{\om\}\times J_\om\sbt E_A^+\times Y,
$$
called the \textit{Smale space} (or the \textit{fibered limit set}) induced by the Smale pre-system $T$. 

For each $\tau\in E_A$ we have  $\hat\pi_2(\tau)\in \ov Y_{\tau|_0^{\infty}}$; so
$
J_\om\sbt \ov Y_\om
$,
for every $\om\in E_A^+$. Since all  $T_\om:Y_\om\to Y_{\sg(\om)}$ are Lipschitz continuous with Lipschitz constant $\l^{-1}$, they extend uniquely to Lipschitz continuous maps  from $\ov Y_\om$ to $\ov Y_{\sg(\om)}$ with  Lipschitz constant $\l^{-1}$. 

\bprop\label{pt1.1}
For every $\om\in E_A^+$ we have that
\beq\label{1t1.1}
T_\om(J_\om)\sbt J_{\sg(\om)},
\eeq
\beq\label{2t1.1}
\bu_{e\in E,  A_{e\om_0}=1}
T_{e\om}(J_{e\om})=J_\om, \ \text{and}
\eeq
\beq\label{3t1.1}
T\circ\hat\pi=\hat\pi\circ\sg
\eeq

\eprop

\bpf
Let $y\in J_\om$;  then $\exists \tau\in E_A(-\infty,-1)$ s.t $A_{\tau_{-1}\om_0}=1$ and $y=\hat\pi_2(\tau\om)$. Then
\beq\label{1_2016_04_01}
\begin{aligned}
\{T_\om(y)\}
&=T_\om\(\bi_{n=1}^{\infty}\ov{T_{\tau\om}^n
(Y_{\tau|_{-n}^{-1}\om})}\) 
\sbt\bi_{n=1}^{\infty}T_\om\(\ov{T_{\tau\om}^n
(Y_{\tau|_{-n}^{-1}\om})}\) 
\sbt\bi_{n=1}^{\infty}\ov{T_\om\(T_{\tau\om}^n
(Y_{\tau|_{-n}^{-1}\om})\)} \\
&=\bi_{n=1}^{\infty}\ov{T_{\tau\om}^{n+1}
(Y_{\tau|_{-n}^{-1}\om})\)}
=\bi_{n=1}^{\infty}\ov{T_{\tau|_{-\infty}^{-1}\om_0(\sg(\om))}^{n+1}
(Y_{\tau|_{-\infty}^{-1}\om_0(\sg(\om))}\)}  \\
&=\hat\pi_2\(\tau|_{-\infty}^{-1}\om_0(\sg(\om))\) \sbt J_{\sg(\om)}
\end{aligned}
\eeq
Thus $T_\om(J_\om)\sbt J_{\sg(\om)}$ meaning that \eqref{1t1.1} holds, and, as  $\{T_\om(y)\}$ and $\big\{\hat\pi_2\(\tau|_{-\infty}^{-1}\om_0(\sg(\om))\)\big\}$, the respective sides of \eqref{1_2016_04_01}, are singletons, we therefore get
\beq\label{4t1.1}
T_\om\hat\pi_2(\tau\om)=\hat\pi_2\circ\sg (\tau\om),
\eeq
meaning that \eqref{3t1.1} holds. The inclusion
$
\bu_{e\in E\atop A_{e\om_0}=1}
T_{e\om}(J_{e\om})\sbt J_\om
$
holds because of \eqref{1t1.1}. In order to prove the opposite one, let $z\in J_\om$. Then $z=\hat\pi(\g\om)$ with some $\g\in E_A(-\infty,-1)$, where $A_{\g_{-1}\om_0}=1$. Formula \eqref{4t1.1} then yields
$$
z
=\hat\pi_2\circ\sg\(\g|_{-\infty}^{-2}\g_{-1}|_{-\infty}^0\om\)
=T_{\g_{-1}\om}\circ\hat\pi_2\(\g|_{-\infty}^{-2}\g_{-1}|_{-\infty}^0\om\)
\in T_{\g_{-1}\om}\(J_{\g_{-1}\om}\).
$$
So,
$
J_\om\sbt \bu_{e\in E\atop A_{e\om_0}=1}
T_{e\om}(J_{e\om}).
$

\epf

\sp\fr Similarly we obtain
$
J_\om=\bu_{\tau\in E_A^n\atop A_{\tau_n\om_0}=1}T_{\tau\om}\(J_{\tau\om}\)
$,
for all $\om\in E_A^+$, and $n >0$. By formula \eqref{1t1.1} we have $T(J)\sbt J$, so consider the system
$$
T:J\to J
$$
which we call the \textit{skew product Smale endomorphism} generated by the Smale system $T:\hat Y\to \hat Y$. By \eqref{2t1.1} it follows

\bobs\label{o3t1.2}
The map $T:J\to J$ is surjective.
\eobs

\bobs\label{o1t1}
If $T:\hat Y\to\hat Y$ is a skew product Smale system, then the following statements are equivalent:

\sp
\begin{itemize}
\item [(a)] For every $\xi\in J$, the fiber $\hat\pi^{-1}(\xi)\sbt E_A$ is compact.

\sp\item [(b)]  For every $y\in Y$, the fiber $\hat\pi_2^{-1}(y)\sbt E_A$ is compact.

\sp\item [(c)] For every $\xi=(\om,y)\in J$, the set
$
\big\{e\in E:A_{e\om_0}=1 \  \and \  \ y\in T_{e\om}(J_{e\om})\big\}
$
is finite.
\end{itemize}
\eobs

\fr If any of these conditions is satisfied, we call the Smale system $T:J\to J$ of \textit{compact type}.

\brem\label{r1t1.2}
In item (a) of Observation~\ref{o1t1} one can replace $J$ by $\hat Y$.
\erem

\bobs\label{o2t1.2}
If for every $y\in Y$ the set
$
\big\{e\in E:A_{e\om_0}=1 \  \and \  \ y\in T_{e\om}(J_{e\om})\big\}
$
is finite for every $\om\in E_A^+$, then  $T:J\to J$ is of compact type.
\eobs

From now on we assume $T:\hat Y\to \hat Y$ is a skew product Smale system of compact type.

If for every $\xi\in\hat Y$ (or in $J$), the fiber $\hat\pi^{-1}(\xi)\sbt E_A$ is finite, we call the skew product Smale system $T$ of \textit{finite type}.

\bobs\label{o2t1}
If the skew product Smale system $T:\hat Y\to \hat Y$ is of finite type, then it is also of compact type. 
\eobs

\fr The Smale system $T:\hat Y\to \hat Y$ is called of \textit{bijective type} if, for every $\xi\in J$ the fiber $\hat\pi^{-1}(\xi)$ is a singleton. Equivalently, the map $\hat\pi:E_A\to J$ is injective; then also $T:J\to J$ is bijective. A Smale skew product of bijective type is clearly of finite type, and thus of compact type. 

\bdfn\label{d1t1.3}
We call a Smale endomorphism continuous if the global map $T:J\to J$ is continuous with respect to the relative topology inherited from $E_A^+\times Y$.
\edfn

\blem\label{l1_2015_11_13}
For every $n\ge 1$ and every $\tau\in E_A(-n,\infty)$, we have that
$$
\hat\pi_2\([\tau]\)=T_\tau^n\(J_\tau\), \ \text{and 
equivalently for every} \  \tau\in E_A, \ 
\hat\pi_2\([\tau|_{-n}^{\infty}]\)=T_\tau^n\(J_{\tau|_{-n}^{\infty}}\)
$$
\elem

\bpf
From \eqref{3t1.1} we get
$
T_\tau^n\(J_{\tau|_{-n}^{\infty}}\)
=T_\tau^n\circ\hat\pi_2\([\tau|_{-n}^{\infty}]|_0^{\infty}\)
=\hat\pi_2\circ\sg^n\([\tau|_{-n}^{\infty}]|_0^{\infty}\)
=\hat\pi_2\([\tau|_{-n}^{\infty}]\).
$
\epf

 As a consequence of  (\ref{3111205}), we get the following

\bobs\label{o1_2015_11_07}
For every $\om\in E_A$, the map 
$
[\om]_0^{\infty}\ni\tau\mapsto \hat\pi_2(\tau)\in J_{\om|_0^\infty}\sbt Y
$
is Lipschitz continuous if $E_A$ is endowed with the metric $d_{\l^{-1}}$. 
\eobs

Note that for every 
$\tau\in E_A^n$, $n\ge 1$, we have \ 
$
\hat\pi([\tau])=\bu_{\om\in[\tau]}\{\om\}\times J_\om.
$

\

Let $M(E_A)$ be the topological space of  Borel probability measures on $E_A$  with the topology of weak convergence, and  $M_\sg(E_A)$ be its closed subspace consisting of $\sg$-invariant measures. Likewise, let $M(J)$ be the space of  Borel probability measures on $J$ with the topology of weak convergence, and let $M_T(J)$ be its closed subspace of $T$-invariant measures. The following fact is well known and easy to prove.

\blem\label{l1t5}
Let $W$ and $Z$ be Polish spaces. Let $\mu$ be a Borel probability measure on $Z$, let $\hat\mu$ be its completion, and denote by $\hat\cB_\mu$ the complete $\sg$-algebra of all $\hat\mu$-measurable subsets of $Z$. Let $f:W\to Z$ be a Borel measurable surjection and let $g:W\to\ov\R$ be a Borel measurable function. Define the functions $g_*, g^*:Z\to\ov\R$ respectively by
$$
g_*(z):=\inf\big\{g(w):w\in f^{-1}(z)\big\} \  \
\and \ \
g_*(z):=\sup\big\{g(w):w\in f^{-1}(z)\big\}.
$$
Then these two functions are measurable with respect to the $\sg$-algebra $\hat\cB_\mu$. If in addition the map $f:W\to Z$ is locally $1$--to--$1$, then both $g_*$ and $g^*:Z\to\ov\R$ are Borel measurable.
\elem


Now we prove the following.

\bthm\label{t1t6}
If $T:J\to J$ is a continuous skew product Smales endomorphism of compact type, then the map
$
M_\sg(E_A)\ni\mu\longmapsto
\mu\circ\hat\pi^{-1}\in M_T(J)
$
is surjective.
\ethm

\bpf
Fix $\mu\in M_T(J)$. Let $\cB_b(E_A)$ and $\cB_b(J)$ be the vector spaces of all bounded Borel measurable real-valued functions defined respectively on $E_A$ and on $J$. Let Let $\cB_b^+(E_A)$ and $\cB_b^+(J)$ be the respective convex cones consisting of non-negative functions. Let 
$$
\hat\cB_b(E_A)
:=\{g\circ\hat\pi:g\in \cB_b(J)\}.
$$
Clearly $\hat\cB_b(E_A)$ is a vector subspace of $\cB_b(E_A)$ and, as $
\hat\pi:E_A\to J$ is a surjection, for each $h\in\hat\cB_b(E_A)$ there exists a unique $g\in \cB_b(J)$ such that $h=g\circ\hat\pi$. Thus, treating, via integration, $\mu$ as a linear functional from $\cB_b(J)$ to $\R$, the formula
$$
\hat\cB_b(E_A)\ni g\circ\hat\pi \longmapsto \hat\mu(g\circ\hat\pi)
:=\mu(g)\in\R,
$$
defines a positive linear functional from $\hat\cB_b(E_A)$ to $\R$. By Lemma~\ref{l1t5} applied to $\hat\pi:E_A\to\R$, for every $h\in \cB_b(E_A)$, the function $h_*\circ\hat\pi:E_A\to\R$ belongs to $\hat\cB_b(E_A)$. Also $h-h_*\circ\hat\pi\ge 0$, thus $h-h_*\circ\hat\pi\in\cB_b^+(E_A)$. Riesz Theorem produces then a positive linear functional $\mu^*:\cB_b(E_A)\to\R$, s.t
$
\mu^*(h)=\hat\mu(h),
$
for every $h\in\hat\cB_b(E_A)$. But $\mu^*$ restricted to the space $C_b(E_A)$ of  bounded continuous functions on 
$E_A$, remains linear and positive. 

\sp{\bf Claim $1^0$:} If $(g_n)_{n=1}^\infty$ is a monotone decreasing sequence of non-negative functions in $C_b(E_A)$ converging pointwise in $E_A$ to the function identically equal to zero, then $\lim_{n\to\infty}\mu^*(g_n)$ exists and is equal to zero.

\bpf
Clearly, $(g_n^*)_{n=1}^\infty$ is a monotone decreasing sequence of non-negative bounded functions that, by Lemma~\ref{l1t5}, belong to $\cB(J)$, thus to $\cB_b^+(J)$. Fix $y\in J$. Since our map $T:J\to J$ is of compact type, the set $\hat\pi^{-1}(y)\sbt E_A$ is compact. Therefore Dini's Theorem applies to let us conclude that the sequence $\(g_n|_{\hat\pi^{-1}(y)}\)_{n=1}^\infty$ converges uniformly to zero. Since all these functions are non-negative, this just means that the sequence $(g_n^*)_{n=1}^\infty$ converges to zero. In conclusion $(g_n^*)_{n=1}^\infty$ is a monotone decreasing sequence of functions in $\cB_b^+(J)$ converging pointwise to zero. Therefore, as also $g_n\le g_n^*\circ\hat\pi$, we get
$$
0
\le \varlimsup_{n\to\infty}\mu^*(g_n)
\le \varlimsup_{n\to\infty}\mu^*(g_n^*\circ\hat\pi)
=\varlimsup_{n\to\infty}\hat\mu(g_n^*\circ\hat\pi)
=\varlimsup_{n\to\infty}\mu(g_n^*)
=0.
$$
So,  $\lim_{n\to\infty}\mu^*(g_n)$ exists and is equal to zero. 
\epf

\fr Having Claim~$1^0$, Daniell-Stone Representation Theorem implies that $\mu^*$ extends uniquely from $C_b(E_A)$ to an element of $M(E_A)$. We denote it also by $\mu^*$. 

\sp{\bf Claim $2^0$:} For every $\ep>0$ there exists $K_\ep$, a compact subset of $
E_A$ such that $\hat\pi^{-1}(\hat\pi(K_\ep))=K_\ep$ and $\mu(\hat\pi(K_\ep))\ge 1-\frac{\ep}2$.
\bpf
Fix $k\in\Z$ and let $p_k:E^{+-}\to E$  the projection on the $k$th coordinate, i.e. 
$
p_k\((\g_n)_{n=-\infty}^{\infty}\)=\g_k.
$ \ 
Fix $\ep>0$. We assume without loss of generality that $E=\{1, 2,\ld\}$. Since  $T:J\to J$ is of compact type, each set $\hat\pi^{-1}(y)\sbt E_A$, $y\in J$, is compact, and thus the function $p_k^*:J\to \ov\R$, defined in Lemma~\ref{l1t5}, takes values in $\R$ and  is Borel measurable. Thus  $p_k^*\circ\hat\pi:
E_A\to\N$ is Borel measurable, and there exists $n_k\ge 1$ so that
\bea\label{2t8}
\mu\((p_k^*)^{-1}([n_k+1,\infty))\)<2^{-|k|-4}\ep.
\eea
Since $\mu$ is inner regular, by Lusin's Theorem, Borel measurability of the function $p_k^*:J\to\N$
yields the existence of closed subsets $J_k\sbt J$ such that 
$$
\mu(J_k)\ge 1-\ep2^{-|k|-4}
\  \  \  {\rm and}  \  \  \  
p_k^*|_{J_k}:J_k\to\N
\  \  {\rm is \ continuous.}
$$
Define
$$
J_\infty:=\bi_{k\in\Z}J_k.
$$
Then $J_\infty$ is a closed subset of $J$, and 
\beq
\label{1t8}
\mu(J_\infty)\ge 1-\frac{\ep}4,
\eeq
and each map $p_k^*|_{J_\infty}: J_\infty\to\N$ is continuous. Define also
$$
K_\ep:=\bi_{k\in\Z}\(p_k^*|_{J_\infty}\circ \hat\pi|_{\hat\pi^{-1}(J_\infty)}\)^{-1}([1,n_k])
$$
By the definition of the maps $p_k^*$ we have that
\beq\label{3t9}
\hat\pi^{-1}(\hat\pi(K_\ep))=K_\ep, \ \text{and} \ \ 
\hat\pi(K_\ep)=J_\infty\cap\bi_{k\in\Z}(p_k^*)^{-1}([1,n_k])
\eeq
Therefore, using \eqref{1t8} and \eqref{2t8}, we get
\beq\label{4t9}
\mu(J\sms\hat\pi(K_\ep))
\le \mu(J\sms J_\infty)+\sum_{k\in\Z}
\mu\((p_k^*)^{-1}([n_k+1,\infty))\)
\le \frac{\ep}4+\frac{\ep}4
=\frac{\ep}2
\eeq
Since $p_k^*|_{J_\infty}, k\in\Z$ are continuous, $K_\ep$ is  closed in $E_A$. Also $K_\ep\sbt \prod_{k\in\Z}[1,n_k]$ and this product is compact, so $K_\ep$ is  compact. Using \eqref{3t9} and \eqref{4t9}, completes the proof.
\epf
Using that $\mu$ is $T$-invariant, and Urysohn's Approximation Method, we show,

\sp{\bf Claim $3^0$:} If $\ep>0$ and $K_\e\sbt E_A$ is the compact set produced in Claim $2^0$, then
$
\mu^*\circ\sg^{-j}(K_\ep)\ge 1-\ep,
$
for all integers $j\ge 0$.
\bpf
Fix $\ep>0$ arbitrary. Fix an integer $j\ge 0$. Since measure $\mu^*\circ\sg^{-j}\circ\hat\pi^{-1}$ is outer regular and $\hat\pi(K_\ep)$ is a Borel (since compact) set, there exists an open set $U\sbt J$ such that
$$
\hat\pi(K_\ep)\sbt U \  \
\and \  \
\mu^*\circ\sg^{-j}\circ\hat\pi^{-1}\(U\sms 
\hat\pi(K_\ep)\)\le \ep/2
$$
Now, Urysohn's Lemma produces a continuous function $u:J\to[0,1]$ such that $u(\hat\pi(K_\ep))=\{1\}$ and $u(E_A\sms U)\sbt \{0\}$. Then, by our construction of $\mu^*$ and by Claim $2^0$, 
$$
\begin{aligned}
\mu^*\circ\sg^{-j}(K_\ep)
&=\mu^*\circ\sg^{-j}\circ \hat\pi^{-1}\(\hat\pi(K_\ep)\)
\ge \mu^*\circ\sg^{-j}\circ\hat\pi^{-1}(U)-
\frac{\ep}2 =\mu^*(\1_U\circ\hat\pi\circ\sg^j)-
\frac{\ep}2\\
&=\mu^*(\1_U\circ T^j)-\frac{\ep}2
\ge \mu^*(u\circ T^j\circ\hat\pi)-
\frac{\ep}2
=\mu(u)-\frac{\ep}2 \ge \mu(\hat\pi(K_\ep))-\frac{\ep}2 \ge 1-\ep
\end{aligned}
$$\epf
Now, for every $n\ge 1$, set
$$
\mu_n^*:=\frac1n\sum_{j=0}^{n-1}
\mu^*\circ\sg^{-j}.
$$
It directly follows from Claim~$3^0$ that
$$
\mu_n^*(K_\ep)\ge 1-\ep,
$$
for every $\ep>0$ and all $n\ge 1$. Also, since, by Claim~$2^0$, each set $K_\ep$ is compact, the sequence of measures $(\mu_n^*)_{n=1}^\infty$ is tight with respect to the weak  topology on 
$M_\sg(E_A)$. There thus exists $(n_k)_{k=1}^\infty$, an increasing sequence of positive integers such that 
$(\mu_{n_k}^*)_{k=1}^\infty$ converges weakly, and denote its limit by $\nu\in M(E_A)$. A standard argument shows that $\nu\in M_\sg(E_A)$. By the definitions of $\hat\mu$ and $\mu^*$, we get
for every $g\in C_b(E_A)$, and every $n\ge 1$, that
$$
\begin{aligned}
\mu_n^*\circ\hat\pi^{-1}(g)
&=\mu_n^*(g\circ\hat\pi)
=\frac1n\sum_{j=0}^{n-1}
\mu^*\circ\sg^{-j}(g\circ\hat\pi)
=\frac1n\sum_{j=0}^{n-1}
\mu^*(g\circ\hat\pi\circ\sg^j) \\
&=\frac1n\sum_{j=0}^{n-1}
\mu^*\((g\circ T^j)\circ\hat\pi\) 
=\frac1n\sum_{j=0}^{n-1}
\hat\mu\((g\circ T^j)\circ\hat\pi\) \\
&=\frac1n\sum_{j=0}^{n-1}\mu(g\circ T^j) =\frac1n\sum_{j=0}^{n-1}\mu(g) =\mu(g)
\end{aligned}
$$
So $\mu_n^*\circ\hat\pi^{-1}=\mu$ for all $n\ge 1$, thus
$
\nu\circ\hat\pi^{-1}
=\mathop{\lim}\limits_{k\to\infty}\mu_{n_k}^*\circ\hat\pi^{-1}
=\mathop{\lim}\limits_{k\to\infty}\mu_{n_k}^*\circ\hat\pi^{-1}
=\mu.
$
\epf

\bobs\label{o1t11}
If $T$ is a Smale endomorphism and $\mu\in M_\sg(E_A)$, then 
$$
\h_{\mu\circ\hat\pi^{-1}}(T)=\hmu(\sg).
$$
\eobs

\bpf
We have two standard inequalities
$
\h_{\mu\circ\hat\pi^{-1}}(T)\le \hmu(\sg)$, and \ 
$\h_{\mu\circ\hat\pi^{-1}\circ \pi_0^{-1}}(\sg)\le\h_{\mu\circ\hat\pi^{-1}}(T)$. 
But $\pi_0: E_A\to E_A^+, \ \pi_0(\tau) = \tau|_0^\infty$ is the canonical projection from $E_A$ to $E_A^+$. So, the measure $\mu\in M_\sg(E_A)$ is the Rokhlin's natural extension of the measure $\mu\circ\hat\pi^{-1}\circ \pi_0^{-1}\in M_\sg(E_A^{+})$. Hence, $\h_{\mu\circ\hat\pi^{-1}\circ \pi_0^{-1}}(\sg)=\hmu(\sg)$. So from the above inequalities,  $\h_{\mu\circ\hat\pi^{-1}}(T)=\hmu(\sg)$.

\epf

Now, we define the topological pressure of continuous real-valued functions on $J$ with respect to the dynamical system $T:J\to J$. Since the space $J$ is \textbf{not compact}, there is no canonical candidate for such definition and we choose the one which will turn out to behave well on the theoretical level (variational principle) and serves well for practical purposes (Bowen's formula). For every finite admissible word $\om\in E_A^{+*}$ let
$$
[\om]_T=:\hat\pi_2([\om]) \sbt J.
$$
If $\psi:J\to\R$ is a continuous function, we define
$$
\P(\psi)=\P_T(\psi)
:=\lim_{n\to\infty}\frac1n\log\sum_{\om\in C^{n-1}}\exp\(\sup(S_n\psi|_{[\om]_T}\),
$$
where 
$
S_n\psi=\sum_{j=0}^{n-1}\psi\circ T^j, \  \  n\ge 1.
$
The limit above exists since the sequence 
$
\N\ni n\longmapsto 
\log\sum_{\om\in C^{n-1}}\exp\(\sup(S_n\psi|_{[\om]_T}\)
$
is subadditive. 
We call $\P_T(\psi)$ the \textit{topological pressure} of the potential $\psi:J\to\R$ with respect to the dynamical system $T:J\to J$. As an immediate consequence of this definition and Definition~\ref{1_2015_11_04}, we get the following.

\bobs\label{o1t12}
If $\psi:J\to\R$ is a continuous function, then
$$
\P_T(\psi)=\P_\sg(\psi\circ\hat\pi).
$$
\eobs

\fr The following theorem follows immediately from Theorem~\ref{t2111105p69}, Observation~\ref{o1t12}, Theorem~\ref{t1t6}, and Observation~\ref{o1t11}, and we will provide such proof.

\bthm\label{t2t12}
If $\psi:J\to\R$ is a continuous function, and $\mu\in M_T(J)$ is such that $\psi \in L^1(J, \mu)$ and $\int\psi\,d\mu>-\infty$, then
$
\hmu(T)+\int_J\psi\,d\mu\le \P_T(\psi).
$
\ethm  

\bpf
By Theorem~\ref{t1t6} there exists $\nu\in M_\sg(E_A)$ such that $\nu\circ\hat\pi^{-1}=\mu$. The other theorems listed immediately above give: \
$$
\hmu(T)+\int_J\psi\,d\mu
=\h_{\nu\circ\hat\pi^{-1}}(T)+\int_J\psi\,d(\nu\circ\hat\pi^{-1})
=\h_\nu(\sg)+\int_{E_A}\psi\circ\hat\pi\,d\hat\nu
\le \P_\sg(\psi\circ\hat\pi)
=\P_T(\psi).
$$
\epf
We have the following two definitions.
\bdfn\label{d3t13}
The measure $\mu\in M_T(J)$ is called an equilibrium state of the continuous potential $\psi:\hat Y\to\R$, if $\int\psi\,d\mu>-\infty$ and 
$$
\hmu(T)+\int_J\psi\,d\mu= \P_T(\psi).
$$
\edfn

\bdfn\label{d3t13a}
The potential $\psi:J\to\R$ is called \textit{summable} if 
$$
\sum_{e\in E}\exp\(\sup(\psi|_{[e]_T})\)<\infty.
$$
\edfn

\bobs\label{o1t13.1}
$\psi:J\to\R$ is summable if and only if $\psi\circ\hat\pi:E_A\to\R$ is summable.
\eobs

\bdfn\label{d1t13}
We call a  continuous skew product Smale endomorphism $T:\hat Y\to\hat Y$ H\"older, if the projection $\hat\pi:E_A\to J$ is H\"older continuous. 
\edfn

\fr We  now  establish an important property of H\"older  skew product Smale endomorphisms of compact type, and then will describe a general construction of such endomorphisms.

\bthm\label{t2t13}
If $T:J\to J$ is H\"older skew product Smale endomorphism of compact type and $\psi:J\to\R$ is a H\"older summable potential, then $\psi$ admits a unique equilibrium state, denoted by $\mu_\psi$. In addition
$
\mu_\psi=\mu_{\psi\circ\hat\pi}\circ\hat\pi^{-1},
$
where $\mu_{\psi\circ\hat\pi}$ is the unique equilibrium state of  $\psi\circ\hat\pi:E_A\to\R$ with respect to  $\sg:E_A\to E_A$.
\ethm

\bpf
 $\psi\circ\hat\pi:E_A\to\R$ is a summable H\"older continuous potential, so it has a unique equilibrium state $\mu_{\psi\circ\hat\pi}$ by Theorem~\ref{t3111105p70}. By Observation~\ref{o1t12} and Observation~\ref{o1t1} we have 
$$
\h_T\(\mu_{\psi\circ\hat\pi}\circ\hat\pi^{-1}\)+
     \int_J\psi\,d\(\mu_{\psi\circ\hat\pi}\circ\hat\pi^{-1}\)
=\h_\sg(\mu_{\psi\circ\hat\pi})+
     \int_{E_A}\psi\circ\hat\pi\,d(\mu_{\psi\circ\hat\pi})
=\P_\sg(\psi\circ\hat\pi)
=\P_T(\psi)
$$
We have to show that, if $\mu$ is an equilibrium measure of  $\psi$, then $\mu=\mu_{\psi\circ\hat\pi}\circ\hat\pi^{-1}$. In this case, from Theorem~\ref{t1t6}, we get $\mu=\nu\circ\hat\pi^{-1}$ for some $\nu\in M_\sg(E_A)$. Then by  Observation~\ref{o1t12},
$$
\h_\nu(\sg)+\int_{E_A}\psi\circ\hat\pi\,d\nu
\ge \h_{\nu\circ\hat\pi^{-1}}(T)+\int_J\psi\,d\(\nu\circ\hat\pi^{-1}\)
=\hmu(T)+\int_J\psi\,d\mu
=\P_T(\psi)
=\P_\sg(\psi\circ\hat\pi)
$$
Hence, $\nu$ is an equilibrium state of  $\psi\circ\hat\pi:E_A\to\R$ and  $\nu=\mu_{\psi\circ\hat\pi}$ (see Theorem~\ref{t3111105p70}). 

\epf

Now we provide the promised construction of H\"older Smale skew product endomorphisms. Start with $(Y,d)$, a complete bounded metric space, and assume given for every $\om\in E_A^+$ a continuous closed injective map
$
T_\om:Y\to Y
$,
satisfying the following conditions 
\beq\label{1t14}
d(T_\om(y_2),T_\om(y_1))\le \l^{-1}d(y_2,y_1),
\eeq
for all $y_1, y_2\in Y$ and some $\l>1$ independent of $\om$,
\beq\label{2t14}
d_\infty(T_\b,T_\a)
:=\sup\big\{d(T_\b(\xi), T_\a(\xi)):\xi\in Y\big\}
\le Cd_\ka(\b,\a)
\eeq
with some constants $C\in (0,\infty)$, $\ka>0$, and all $\a,\b\in E_A^+$. \ \ Then
$$
\hat Y= E_A^+\times Y,
$$
and call $T:\hat Y\to \hat Y$ a skew product Smale system of \textit{global character}. We may assume without loss of generality that
\beq\label{1t15}
\ka\le \frac12\log\l.
\eeq

\bthm\label{t1t15}
Each skew product Smale system of global character is H\"older.
\ethm

\bpf
Let $T:E_A^+\times Y\to E_A^+\times Y$ be such skew product Smale system. We first show that $T:E_A^+\times Y\to E_A^+\times Y$ is continuous. Enough to show that $p_2\circ T: E_A^+\times Y \to Y$ is continuous, with $p_2$ the projection to second coordinate. For all $\a,\b\in E_\om^+$ and $z, w\in Y$, 
$$
\begin{aligned}
d(p_2\circ T(\a,z),p_2\circ T(\b,w))
&= d(T_\a(z), T_\b(w))
 \le d(T_\a(z), T_\b(z))+d(T_\b(z), T_\b(w)) \\ 
 &\le d_\infty(T_\a,T_\b)+\l^{-1}d(z,w) \le Cd_\ka(\a,\b)+\l^{-1}d(z,w),
\end{aligned}
$$
and continuity of the map $p_2\circ T:E_A^+\times Y\to Y$ is proved. So the continuity of $T:E_A^+\times Y\to E_A^+\times Y$ is proved, and thus $T:J\to J$ is continuous too. We now show that $T:J\to J$ is H\"older. So, fix an integer $n\ge 0$, two words $\a,\b\in E_A$ and $\xi\in Y$. We then have
\beq\label{2t15}
\begin{aligned}
  d\(T_\a^{n+1}(\xi)&,T_\b^{n+1}(\xi)\) 
  =d\lt(T_\a^n\(T_{\a|_{-(n+1)}^{\infty}}(\xi)\),T_\b^n\(T_{\b|_{-(n+1)}^{\infty}}(\xi)\)\rt) \\
&\le d\lt(T_\a^n\(T_{\a|_{-(n+1)}^{\infty}}(\xi)\),T_\a^n\(T_{\b|_{-(n+1)}^{\infty}} 
      (\xi)\)\rt)
+d\lt(T_\a^n\(T_{\b|_{-(n+1)}^{\infty}}(\xi)\),T_\b^n\(T_{\b|_{-(n+1)}^{\infty}}(\xi)\)\rt) \\
&\le \l^{-n}d\(T_{\a|_{-(n+1)}^{\infty}}(\xi)\),T_{\b|_{-(n+1)}^{\infty}}(\xi)\)
     +d_\infty\(T_\a^n,T_\b^n\)\\
     &\le \l^{-n}Cd_\ka\(\a|_{-(n+1)}^{\infty},\b|_{-(n+1)}^{\infty}\)
    +d_\infty\(T_\a^n,T_\b^n\)
\end{aligned}
\eeq
Let $p\ge -1$ be uniquely determined by the property that 
\beq\label{3t15}
d_\ka(\a,\b)=e^{-\ka p}.
\eeq
Consider two cases. First assume that
$
d_\ka(\a,\b)\ge e^{-\ka n}.
$
Then using also \eqref{1t15}, we get
\beq\label{5t15}
\l^{-n}d_\ka\(\a|_{-(n+1)}^{\infty},\b|_{-(n+1)}^{\infty}\)
\le e^{-2\ka n}
\le e^{-\ka n}d_\ka(\a,\b)
\eeq
So, assume that
\beq\label{3t15B}
d_\ka(\a,\b)< e^{-\ka n}
\eeq
Then $n<p$, so $n+1\le p$, whence 
$$
\begin{aligned}
d_\ka\(\a|_{-(n+1)}^{\infty},\b|_{-(n+1)}^{\infty}\)
&=e^{-\ka(n+2)}e^{-\ka p} =e^{-\ka(n+2)}d_\ka(\a,\b) \le e^{-\ka n}d_\ka(\a,\b)
\end{aligned}
$$
Hence,
$$
\l^{-n}d_\ka\(\a|_{-(n+1)}^{\infty},\b|_{-(n+1)}^{\infty}\)
\le e^{-\ka n}d_\ka(\a,\b)
$$
Inserting this and \eqref{5t15} to \eqref{2t15} in either case yields
$$
d\(T_\a^{n+1}(\xi),T_\b^{n+1}(\xi)\)
\le d_\infty\(T_\a^n,T_\b^n\)+Ce^{-\ka n}d_\ka(\a,\b)
$$
Taking supremum over all $\xi\in Y$, we get
$
d_\infty\(T_\a^{n+1},T_\b^{n+1}\)
\le d_\infty\(T_\a^n,T_\b^n\)+Ce^{-\ka n}d_\ka(\a,\b).
$
Thus, by induction 
\beq\label{1t18}
d_\infty\(T_\a^n,T_\b^n\)
\le Cd_\ka(\a,\b)\sum_{j=0}^{n-1}e^{-\ka n}
\le Cd_\ka(\a,\b)\sum_{j=0}^\infty e^{-\ka n}
=C(1-e^{-\ka})^{-1}d_\ka(\a,\b)
\eeq
for all $\a,\b\in E_A$ and all integers $n\ge 0$. Recall that the integer $p\ge -1$ is determined by \eqref{3t15}. Assume that $p\ge 0$. Then using \eqref{1t18}, \eqref{3t15B}, and \eqref{3111205}, we get
$$
\begin{aligned}
d(\hat\pi_2(\a),(\hat\pi_2(\a))
&\le \diam\(T_\a^p(Y)\)+\diam\(T_\b^p(Y)\)+d_\infty\(T_\a^p,T_\b^p\) \\
&\le \l^{-p}\diam(Y)+\l^{-p}\diam(Y)+\frac{C}{1-e^{-\ka}}d_\ka(\a,\b) \\
&\le 2\diam(Y)d_\ka^{\frac{\log\l}{\ka}}(\a,\b)+\frac{C}{1-e^{-\ka}}d_\ka(\a,\b)
\end{aligned}
$$
Thus $\hat\pi_2:E_A\to Y$ is H\"older continuous, so
$\hat\pi:E_A\to Y$ is H\"older continuous. 

\epf

\section{Conformal Skew Product Smale Endomorphisms}\label{CSPSE}
In this section we keep the setting of skew product Smale endomorphisms. However we assume more about the spaces $Y_\om$, $\om\in E_A^+$, and the fiber maps $T_\om:Y_\om\to Y_{\sg(\om)}$, namely:

\sp\begin{itemize}
\item[(a)] $Y_\om$ is a closed bounded subset of $\R^d$, with some $d\ge 1$ such that $\ov{\Int(Y_\om)}=Y_\om$. 

\sp\item[(b)] Each map $T_\om:Y_\om\to Y_{\sg(\om)}$ extends to a $C^1$ conformal embedding from $Y_\om^*$ to $Y_{\sg(\om)}^*$, where $Y_\om^*$ is a bounded connected open subset of $\R^d$ containing $Y_\om$. The same symbol $T_\om$ denotes this extension and we assume that  $T_\om:Y_\om^*\to Y_{\sg(\om)}^*$ satisfy:

\sp\item[(c)] Formula \eqref{1111205}  holds for all $y_1, y_2\in Y_\om^*$, perhaps with some smaller constant $\l>1$.

\sp\item[(d)] (Bounded Distortion Property 1) There exist constants $\a>0$ and $H>0$ such that for all $y, z\in Y_\om^*$ we have that: \ 
$$
\big|\log|T_\om'(y)|-\log|T_\om'(z)|\big|\le H||y-z||^\a.
$$
\item[(e)] The function 
$
E_A\ni\tau\longmapsto\log|T_\tau'(\hat\pi_2(\om))|\in\R
$
is H\"older continuous.

\sp \item[(f)] (Open Set Condition) For every $\om\in E_A^+$ and for all $a, b\in E$ with $A_{a\om_0}=A_{b\om_0}=1$ and $a\ne b$, we have \ 
$$
T_{a\om}(\Int(Y_{a\om}))\cap T_{b\om}(\Int(Y_{b\om}))=\es.
$$
\item[(g)] (Strong Open Set Condition) There exists a measurable function $\d:E_A^+\to(0,\infty)$ such that for every $\om \in E_A^+$, \ 
$$
J_\om\cap\(Y_\om\sms\ov B(Y_\om^c,\d(\om)\)\ne\es.
$$
\end{itemize}

\fr Any skew product Smale endomorphism satisfying conditions (a)--(g) will be called in the sequel a \textit{conformal skew product Smale endomorphism}. 

\

A standard calculation based on (c), (d), and (e), yields the following:

\sp\fr (BDP2) (Bounded Distortion Property 2)  For some constant $H$, we have that
$$
\Big|\log\big|\(T_\tau^n\)'(y)\big|-
\log\big|\(T_\tau^n\)'(z)\big|\Big|
\le H||y-z||^\a.
$$
for all $\tau\in E_A,  \ y,z\in Y_{\tau|_{-n}^{\infty}}^*$, and all $n >0$.

\sp\fr An immediate consequence of (BDP2) is the following version.

\sp\fr (BDP3) (Bounded Distortion Property 3) For all $\tau\in E_A$, all $n\ge 0$, and all $y,z\in Y_{\tau|_{-n}^{\infty}}^*$, if $K:=\exp\(H\diam^\a(Y)\)$, then we have that
$$
K^{-1}
\le \frac{\big|\(T_\tau^n\)'(y)\big|}
{\big|\(T_\tau^n\)'(z)\big|}
\le K
$$

\brem\label{r1ct1}
Bounded Distortion Property 1, (d), is always satisfied if $d\ge 2$. If $d=2$, this is due of Koebe Distortion Theorem since each conformal map in $\C$ is either holomorphic or antiholomorphic. If $d\ge 3$ this follows from Liouville Representation Theorem saying that conformal maps in $\R^d$, $d\ge 3$, are either  M\"obius transformations or similarities.
\erem

\

Recall also that we say that a conformal skew product Smale endomorphism is \textit{H\"older}, if the condition of H\"older continuity for $\hat\pi: E_A \to J$ is satisfied, see Definition \ref{d1t13}.


As an immediate consequence of the Open Set Condition (f) we get the following.

\blem\label{l1ct3}
Let $T:\hat Y\to \hat Y$ a conformal skew product Smale endomorphism. If $n\ge 1$, $\a, \b\in E_A(-n,\infty)$, $\a|_0^{\infty}=\b|_0^{\infty}$, and $\a|_{-n}^{-1}\ne \b|_{-n}^{-1}$, then
$$
T_\a^n(\Int(Y_\a))\cap T_\b^n(\Int(Y_\b)) =\es.
$$
In fact we have more: 
$$
T_\a^n(\Int(Y_\a))\cap T_\b^n(Y_\b) 
=\es=
T_\a^n(Y_\a)\cap T_\b^n(\Int(Y_\b)).
$$
\elem

\blem\label{l2ct3}
Let $T:\hat Y\to \hat Y$ be a conformal skew product Smale endomorphism. If $n\ge 1$ and $\tau\in E_A(-n,\infty)$, then
$$
\hat\pi_2^{-1}\(T_\tau^n(\Int(Y_\tau))\)
\sbt[\tau].
$$
\elem

\bpf
Let $\g\in \hat\pi_2^{-1} \(T_\tau^n(\Int(Y_\tau))\)$, hence $\g|_0^{\infty}=\tau|_0^{\infty}$ and $\hat\pi_2(\g)\in T_\tau^n(\Int(Y_\tau))\sbt Y_{\tau|_0^{\infty}}$. Also, $\hat\pi_2(\g)\in T_{\g|_{-n}^{\infty}}^n \(Y_{\g|_{-n}^{\infty}}\)$. From  Lemma~\ref{l1ct3} it follows that $\g|_{-n}^0=\tau$, so $\g\in[\tau]$.
\epf
 We will also use the following:

\sp (h) \ (Uniform Geometry Condition) \ $\exists(R>0)\,\, \forall(\om\in E_A^+)\,\, \exists(\xi_\om\in Y_\om)$
$$
B(\xi_\om,R)\sbt Y_\om.
$$

\fr The primary significance of Uniform Geometry Condition (h) lies in:

\blem\label{l1ct3A}
If $T:\hat Y\to\hat Y$ is a H\"older conformal skew product Smale endomorphism satisfying Uniform Geometry Condition (h), then for every  $\ga\ge 1$, $\exists$ $\Ga_\ga>0$ such that:

\sp If $\cF\sbt E_A^*(-\infty,-1)$ is a collection of mutually incomparable (finite)  words, so that $A_{\tau_{-1}\om_0}=1$ for some $\om\in E_A^*$ and all $\tau\in\cF$, and so that for some $\xi\in Y_\om$,
$$
T_{\tau\om}^{|\tau|}(Y_{\tau\om})\cap B(\xi,r)\ne\es  
\  \  with  \  
\ga^{-1}r \le \diam\(T_{\tau\om}^{|\tau|}(Y_{\tau\om})\)\le \ga r,
$$
 then the cardinality of $\cF$ is bounded above by $\Ga_\ga$. 
\elem
\bpf
The family $\{T_{\tau\om}^{|\tau|}(\Int(Y_{\om\tau})):\tau\in\cF\}$ consists of mutually disjoint sets in $Y_\om$. 
We get
$$
\begin{aligned}
T_{\tau\om}^{|\tau|}(\Int(Y_{\tau\om}))
&\spt T_{\tau\om}^{|\tau|}(B(\xi_{\tau\om}, R))
 \spt B\(T_{\tau\om}^{|\tau|}(\xi_{\tau\om}, K^{-1}R\big|\(T_{\tau\om}^{|\tau|}\)'(\xi_{\tau\om})\big|\) \spt B\(T_{\tau\om}(\xi_{\tau\om}), K^{-2 }R\ga^{-1}r\),
\end{aligned}
$$
from the Uniform Geometry condition.
Also $T_{\tau\om}^{|\tau|}(\Int(Y_{\om\tau}))\sbt B(\xi, (1+\ga)r)$. 

\epf

\section{Volume Lemmas}\label{VL}

We keep the setting of Section~\ref{CSPSE}, with $T:\hat Y\to \hat Y$ a conformal skew product Smale endomorphism, i.e. satisfying conditions (a)--(g) of Section~\ref{CSPSE}. 

\bdfn\label{EDim}
We say in general that  a measure $\mu$ is exact dimensional on a space $X$, if its \textit{pointwise dimension} at $x$, defined by the formula $$d_\mu(x) = \mathop{\lim}\limits_{r\to 0} \frac{\log \mu(B(x, r))}{\log r},$$ exists for $\mu$-a.e $x\in X$, and $d_\mu(\cdot)$ is constant $\mu$-almost everywhere (see \cite{Pe}).
\edfn

 If $\mu$ is a Borel probability $\sg$-invariant measure on $E_A$, then by $\chi_\mu(\sg)$ we denote its \textit{Lyapunov exponent}, defined by the formula,
$$
\chi_\mu(\sg)
:=-\int_{E_A}\log\Big|T_{\tau|_0^{\infty}}'\(\hat\pi_2(\tau)\)\Big|\,d\mu(\tau)
 =-\int_{E_A^+}\int_{[\om]}\log\big|T_\om'\(\hat\pi_2(\tau)\)\big|\,d\ov\mu^\om
   (\tau)\,dm(\om),
$$
where $m=\mu\circ \pi_0^{-1}= \pi_{1*}\mu$ is the canonical projection of $\mu$ onto $E_A^+$. 

\

We prove now the exact dimensionality of projections of conditional measures for equilibrium states, onto the fibers: 

\bthm\label{t1vl4}(Exact dimensionality of conditional measures in fibers).

Let $T:\hat Y\to\hat Y$ be a H\"older conformal skew product Smale endomorphism, and let $\psi:E_A\to\R$ be a H\"older continuous summable potential. Then for the projection $\hat\pi_{2*} \ov\mu_\psi^\om = \ov\mu_\psi^\om\circ\hat\pi_2^{-1}$, of the conditional measure to the fiber $J_\om$, we have 
$$
\HD\(\ov\mu_\psi^\om\circ\hat\pi_2^{-1}\)
=\frac{\h_{\mu_\psi}(\sg)}{\chi_{\mu_\psi}(\sg)}
=\frac{\P_\sg(\psi)-\int\psi\,d\mu_\psi}{\chi_{\mu_\psi}(\sg)}
$$
for $m_\psi$-a.e $\om\in E_A^+$, where $m_\psi=\mu_\psi\circ \pi_0^{-1}$. Moreover for $m_\psi$-a.e $\om\in E_A^+$ the measure $\ov\mu_\psi^\om\circ\hat\pi_2^{-1}$ is  exact dimensional, and its pointwise dimension is given by:
\beq\label{1vl4}
\lim_{r\to 0}
\frac
{\log\ov\mu_\psi^\om\circ\hat\pi_2^{-1}
(B,r))}{\log r}
=\frac{\h_{\mu_\psi}(\sg)}{\chi_{\mu_\psi}(\sg)},
\eeq
for $m_\psi$-a.e. $\om\in E_A^+$ and $\ov\mu_\psi^\om\circ\hat\pi_2^{-1}$-a.e. $z\in J_\om$. 
\ethm 

\bpf
We only need to show that \eqref{1vl4} holds. Since $\mu_\psi$ is ergodic, Birkhoff's Ergodic Theorem applied to  $\sg^{-1}:E_A\to E_A$ gives a measurable set $E_{A, \psi}\sbt E_A$ s.t $\mu_\psi(E_{A, \psi})=1$,
\beq\label{2vl4}
\lim_{n\to\infty}\frac1n\log\big|\(T_\tau^n\)'(\hat\pi_2(\sg^{-n}(\tau)))\big|
=-\chi_{\mu_\psi}(\sg), \ \text{and} \  
\lim_{n\to\infty}\frac1n S_n\psi(\sg^{-n}(\tau))
=\int_{E_A}\psi\,d\mu_\psi
\eeq
for every  $\tau\in E_{A, \psi}$. For arbitrary $\om\in E_A^+$ denote now:
$$
\nu_\om:=\ov\mu_\psi^\om\circ\hat\pi_2^{-1},
$$
which is a Borel probability measure on $J_\om$. Fix $\tau\in E_{A, \psi}$. Fix a radius $r\in \(0,\diam\(Y_{p_2(\tau)}\)/2\)$. Let $z=\hat\pi_2(\tau)$, and consider the least integer $n=n(z,r)\ge 0$ so that
\beq\label{1vl6}
T_\tau^n\(Y_{\tau|_{-n}^{\infty}}\)
\sbt B(z,r).
\eeq
If $r>0$ is small enough (depending on $\tau$), then $n\ge 1$ and 
$
T_\tau^{n-1}\(Y_{\tau|_{-(n-1)}^{\infty}}\)
\not\sbt B(z,r).
$
Since $z\in T_\tau^{n-1}\(Y_{\tau|_{-(n-1)}^{\infty}}\)$, this implies that
\beq\label{2vl6}
\diam\lt(T_\tau^{n-1}\(Y_{\tau|_{-(n-1)}^{\infty}}\)\rt)
\ge r.
\eeq
Write $\om:=\tau|_0^{\infty}$. It follows from \eqref{1vl6}, Lemma~\ref{l1_2015_11_13}, and Theorem~\ref{t4111105p70} that
\beq\label{3vl6A}
\begin{aligned}
\nu_\om(B(z,r))
&\ge \nu_\om\(\hat\pi_2\([\tau|_{-n}^{\infty}]\)\)
 =   \ov\mu_\psi^\om\circ\hat\pi_2^{-1}\(\hat\pi_2\([\tau|_{-n}^{\infty}]\)\)
 \ge \ov\mu_\psi^\om\([\tau|_{-n}^{\infty}]\) \\
&\ge D^{-1}\exp\(S_n\psi(\sg^{-n}(\tau))-\P_\sg(\psi)n\).
\end{aligned}
\eeq
By taking logarithms and using \eqref{2vl6}, this gives that
$$
\frac{\log\nu_\om(B(z,r))}{\log r}
\le \frac{-\log D+S_n\psi(\sg^{-n}(\tau))-\P_\sg(\psi)n}
{\log\lt(\diam\lt(T_\tau^{n-1}\(Y_{\tau|_{-(n-1)}^{\infty}}\rt)\rt)}
$$
So applying (BDP3), we get that
$$
\begin{aligned}
\frac{\log\nu_\om(B(z,r))}{\log r}
&\le \frac{-\log D+S_n\psi(\sg^{-n}(\tau))-\P_\sg(\psi)n}
    {\log K+\log\lt(\diam\lt(Y_{\tau|_{-(n-1)}^{\infty}}\rt)\rt)
    +\log \lt|\lt(T_\tau^{n-1}\rt)'(\hat\pi_2\(\sg^{-n}(\tau))\)\rt|} 
\end{aligned}
$$
so by dividing both numerator and denominator by $n$, and using that $\diam\lt(Y_{\tau|_{-(n-1)}^{\infty}}\rt)=\diam(Y)$ and \eqref{2vl4}, this yields,
\beq\label{3vl6}
\varlimsup_{r\to 0}\frac{\log\nu_\om(B(z,r))}{\log r}
\le \frac {\lim_{n\to\infty}\frac1nS_n\psi(\sg^{-n}(\tau))-\P_\sg(\psi)}
          {\lim_{n\to\infty}\frac1n \log\lt|\lt(T_\tau^{n-1}\rt)'(\hat\pi_2\(\sg^{-n} 
           (\tau))\)\rt|}
=\frac{\P_\sg(\psi)-\int\psi\,d\mu_\psi}{\chi_{\mu_\psi}(\sg)}.
\eeq
To prove the opposite inequality, note that the set $\hat\pi_2^{-1}\(J_\om \sms \ov B(Y^c_\om,\d(\om))\)$ is open in $[\om]\sbt E_A$, it is not empty by (g), and thus
$$
\ov\mu_\psi^\om\(\hat\pi_2^{-1}\(J_\om \sms \ov B(Y^c_\om,\d(\om))\)\)>0
$$
for every $\om\in E_A^+$. Consequently, $\mu_\psi(Z)>0$, where
$
Z:=\bu_{\om\in E_A^+}\hat\pi_2^{-1}\(J_\om \sms \ov B(Y^c_\om,\d(\om))\)
$. \ 
Since $\d:E_A^+\to(0,\infty)$ is measurable, there exists $R>0$ s.t
$
\mu_\psi(Z_R)>0,
$
where
$$
Z_R:=\bu_{\om\in E_A^+}\hat\pi_2^{-1}\(J_\om \sms \ov B(Y^c_\om,R)\)
$$
Consider the set
$
N(\tau):=\{k\ge 0:\sg^{-k}(\tau)\in Z_R\}.
$ \
Represent this set $N(\tau)$ as a strictly increasing sequence $(k_n(\tau))_{n=1}^\infty$. By Birkhoff's Ergodic Theorem there is a measurable set $\tilde E_{A, \psi}\sbt E_{A, \psi}$ with $\mu_\psi(\tilde E_{A, \psi})=1$ and
for every $\tau' \in \tilde E_{A, \psi}$, $$\lim_{n \to \infty} \frac{Card\{0 \le i \le n, \ \sigma^{-i}(\tau') \in Z_R\}}{n} = \mu_\psi(Z_R)$$
Now we put $k_n(\tau) \ge n$, instead of $n$ above, and notice that $Card\{0 \le i \le k_n(\tau), \ \sigma^{-i}(\tau') \in Z_R\} = n$. Therefore as $\mu_\psi(Z_R) >0$, we obtain for every $\tau \in \tilde E_{A, \psi}$ and any $n$ large, that:
$$\lim_{n \to \infty}\frac{k_n(\tau)}{n} = \frac{1}{\mu_\psi(Z_R)}$$
Hence for every $\tau \in \tilde E_{A, \psi}$, 
\beq\label{1vl7.1}
\lim_{n\to\infty}\frac{k_{n+1}(\tau)}{k_n(\tau)}=1
\eeq
Fix $\tau\in \tilde E_{A, \psi}$, $\om=\tau|_0^{\infty}$, and let the largest $n=n(\tau,r)\ge 1$ s.t with $k_j:=k_j(\tau)$, $j\ge 1$, 
\beq\label{1vl7}
K^{-1}\lt|\lt(T_\tau^{k_n}\rt)'\(\hat\pi_2(\sg^{-k_n}(\tau))\)\rt|R\ge r.
\eeq
Then 
\beq\label{2vl7}
K^{-1}\lt|\lt(T_\tau^{k_{n+1}}\rt)'\(\hat\pi_2(\sg^{-k_{n+1}}(\tau))\)\rt|R<r.
\eeq
It follows from \eqref{1vl7} and (BDP3) that
$
B(z,r)
\sbt T_\tau^{k_n}\(B(\hat\pi_2(\sg^{-k_n}(\tau)),R)\)
\sbt T_\tau^{k_n}\lt(\Int\lt(Y_{\tau|_{-k_n}^{\infty}}\rt)\rt).
$
Hence, invoking also Lemma~\ref{l2ct3} and Theorem~\ref{t4111105p70}, we infer that 
$$
\nu_\om(B(z,r))
\le \ov\mu_\psi^\om\([\tau|_{-k_n}^{\infty}]\)
\le D\exp\lt(S_{k_n}\psi(\sg^{-k_n}(\tau))-\P_\sg(\psi)k_n\rt).
$$
By taking logarithms and using \eqref{2vl7}, this gives 
$$
\frac{\log\nu_\om(B(z,r))}{\log r}
\ge\frac{\log D+S_{k_n}\psi(\sg^{-k_n}(\tau))-\P_\sg(\psi)k_n}
    {-\log K+\log\lt|\lt(T_\tau^{k_{n+1}}\rt)'(\hat\pi_2\(\sg^{-k_{n+1}}(\tau))\)\rt|}
$$
Dividing both numerator and denominator above by $k_n$, and using  \eqref{2vl4}, \eqref{1vl7.1}, it yields
$$
\varliminf_{r\to 0}\frac{\log\nu_\om(B(z,r))}{\log r}
\ge \frac{\lim_{n\to\infty}\frac1{k_n}S_{k_n}\psi(\sg^{-k_n}(\tau))-\P_\sg(\psi)}
         {\lim_{n\to\infty}\frac1{k_n}
         \log\lt|\lt(T_\tau^{k_{n+1}}\rt)'(\hat\pi_2\(\sg^{-k_{n+1}}(\tau))\)\rt|}
=\frac{\P_\sg(\psi)-\int\psi\,d\mu_\psi}{\chi_{\mu_\psi}(\sg)}.
$$
From \eqref{3vl6}, it follows that \eqref{1vl4} holds for all $\tau\in \tilde E_{A, \psi}$.

\epf

 If $\mu$ is now a Borel probability $T$-invariant measure on the fibered limit set $J$, then by $\chi_\mu(T)$ we again denote its \textit{Lyapunov exponent}, defined by 
$$
\chi_\mu(T)
:=-\int_J\log\big|T_\om'(z)\big|\,d\mu(\om,z)
 =-\int_{E_A^+}\int_{J_\om}\log\big|T_\om'(z)\big|\,d\ov\mu^\om(z)\,dm(\om),
$$
where $m=\mu\circ \pi_0^{-1}$ is the projection of $\mu$ onto $E_A^+$, and $\(\ov\mu^\om\)_{\om\in E_A^+}$ is the canonical system of conditional measures of $\mu$ for the measurable partition $\{\{\om\}\times J_\om\}_{\om\in E_A^+}$. 

\fr Now we prove:

\bcor\label{c1vl10}
Let $T:\hat Y\to\hat Y$ be a H\"older conformal Smale endomorphism of compact type. Let $\psi:J\to\R$ be a H\"older continuous summable potential. Then
$$
\HD\(\ov\mu_\psi^\om\)
=\frac{\h_{\mu_\psi}(T)}{\chi_{\mu_\psi}(T)}
=\frac{\P_T(\psi)-\int\psi\,d\mu_\psi}{\chi_{\mu_\psi}(T)}
$$
for $m_\psi$-a.e. $\om\in E_A^+$, where $m_\psi=\mu_\psi\circ p_1^{-1}$. Moreover, for $m_\psi$-a.e. $\om\in E_A^+$ the measure $\ov\mu_\psi^\om$ is dimensional exact, \ and for $m_\psi$-a.e. $\om\in E_A^+$ and $\ov\mu_\psi^\om$-a.e. $z\in J_\om$,
\beq\label{1vl4A}
\lim_{r\to 0}
\frac
{\log\ov\mu_\psi^\om(B(z, r))}{\log r}
=\frac{\h_{\mu_\psi}(T)}{\chi_{\mu_\psi}(T)}
\eeq
\ecor

\bpf
Let $\hat\psi:=\psi\circ\hat\pi:E_A\to\R$. By Theorem~\ref{t2t13}
$\mu_\psi=\mu_{\hat\psi}\circ\hat\pi^{-1}$ is the unique equilibrium state of the potential
$\psi$ and the shift map $\sg:E_A\to E_A$. By Observation~\ref{o1t12},
 $\P_T(\psi)=\P_\sg(\hat\psi)$, and by Observation~\ref{o1t11}, $\h_{\mu_\psi
 }(T)=\h_{\mu_{\hat\psi}}(\sg)$. Since in addition $\chi_{\mu_\psi}(T)=
 \chi_{\mu_{\hat\psi}}(\sg)$, the proof follows immediately from Theorem~\ref{t1vl4} applied to  $\hat\psi:E_A\to\R$.
\epf

Uniform Geometry Condition (h) was not required in this section; it will be used in the next one.

\section{Bowen's Formula}\label{bowen}

We keep the setting of Sections ~\ref{CSPSE} and Section~\ref{VL}, so $T:\hat Y\to \hat Y$ is a conformal skew product Smale endomorphism, i.e. satisfies conditions (a)--(g) of Section~\ref{CSPSE}. We however emphasize that in Section \ref{bowen}, Condition (h) i.e. the Uniform Geometry Condition, is assumed.

\sp For every $t\ge 0$ let $\psi_t:J\to\R$ be the function
$$
\psi_t(\om,y)=-t\log|T_\om'(y)|.
$$
Define $\cF(T)$ to be the set of parameters $t\ge 0$ for which the potential $\psi_t$ is summable, i.e.
$$
\sum_{e\in E}\exp\(\sup\(\psi_t|_{[e]_T}\)\)<\infty.
$$
This means that
$$
\sum_{e\in E}\sup\big\{||T_{e\tau}||_\infty^t:\tau\in E_A(1,\infty), \, A_{e\tau_1}=1\big\}
<\infty.
$$
For every $t\ge 0$, let
$$
\P(t):=\P_T(\psi_t),
$$
and call $\P(t)$ the topological pressure of the parameter $t$. From Proposition~\ref{p2_2015_11_14}, we have
$
\cF(T)=\{t\ge 0:\P(t)<\infty\}.
$
We record the following basic properties of this pressure.

\bprop\label{p1bf1}
The pressure function $t\mapsto\P(t), \ t \in [0, \infty)$ has the following properties:

\sp\begin{itemize}
\item[(a)] $\P$ is monotone decreasing

\sp\item[(b)] $\P|_{\cF(T)}$ is strictly decreasing.

\sp\item[(c)] $\P|_{\cF(T)}$ is convex, real-analytic, and Lipschitz continuous.
\end{itemize}
\eprop

\bpf
All these statements except real analyticity follow easily from definitions, plus, due to Lemma~\ref{l3_2015_11_17} and Observation~\ref{o1t12}, from their one-sided shift counterparts.
\epf 

\sp Now we can define two significant numbers associated with the  Smale endomorphism $T$: 
$$
\th_T:=\inf\big\{t\ge 0:\P(t)<\infty\big\} 
\  \and \   
B_T:=\inf\big\{t\ge 0:\P(t)\le 0\big\}.
$$
The number $B_T$ is called the \textit{Bowen's parameter} of the system $T$. Clearly
$
\th_T\le B_T.
$

\bthm\label{t1bf2}
If $T:\hat Y\to\hat Y$ is a H\"older conformal skew product Smale endomorphism satisfying the Uniform Geometry Condition (h), then for every $\om\in E_A^+$,
$$
\HD(J_\om)=B_T
$$
\ethm

\fr We first prove this theorem under the assumption that the alphabet $E$ is finite. In this case we will actually prove more.  Recall that if $(Z,\rho)$ is a separable metric space, then a finite Borel measure $\nu$ on $Z$ is called \textit{Ahlfors regular} (or \textit{geometric}) if and only if
$$
C^{-1}r^h\le \nu(B(z,r))\le Cr^h,
$$ 
for all $r >0$, with some independent constants $h\ge 0$, $C \in(0,\infty)$. It is well known and easy to prove that there is at most one $h$ with such property and all Ahlfors regular measures on $Z$ are mutually equivalent, with bounded Radon-Nikodym derivatives. Moreover
$$
h=\HD(Z)=\PD(Z)=\BD(Z),
$$
the two latter being respectively the packing and box-counting dimensions of $Z$. 

Now, if the alphabet $E$ is finite, then the Smale endomorphism $T:\hat Y\to\hat Y$ is of compact type, and in particular, for every $t\ge 0$ there exists $\mu_t$, a unique equilibrium state for the potential $\psi_t:J\to\R$. Since $0\le \P(0)<\infty$ it follows from Proposition~\ref{p1bf1} that
$
\P(B_T)=0.
$

\bthm\label{t1bf3}
If $T:\hat Y\to\hat Y$ is a H\"older conformal skew product Smale endomorphism satisfying the Uniform Geometry Condition (h) and the  alphabet $E$ is finite, then $\ov\mu_{B_T}^\om$ is an Ahlfors regular measure on $J_\om$, for every $\om\in E_A^+$. In particular, for every $\om\in E_A^+$,
$$
\HD(J_\om)=B_T
$$

\ethm

\bpf
Put 
$
h:=B_T.
$
Fix $\om\in E_A^+$ and $z=\hat\pi_2(\tau)\in J_\om$ arbitrary. Let $n=n(z,r)$ be given by \eqref{1vl6}, and denote
$
\nu_\om:=\ov\mu_h^\om\circ\hat\pi_2^{-1}.
$
The formula \eqref{3vl6A} gives, for $\psi=\psi_h$,
\beq\label{1bf4}
\nu_\om(B(z,r))
\ge D^{-1}\exp\(S_n\psi(\sg^{-n}(\tau))\)
=D^{-1}\big|\(T_\tau^n\)'(\hat\pi_2(\sg^{-n}(\tau)))\big|^h.
\eeq
Now, since  $E_A$ is compact (as $E$ is finite) and since  $E_A
\ni\tau\mapsto|T_\tau'(\hat\pi_2(\tau))|\in (0,\infty)$ is continuous, we conclude that there exists a constant $M\in(0,\infty)$ such that
\beq\label{2bf4}
M^{-1}
\le \inf\big\{|T_\tau'(\hat\pi_2(\tau))|:\tau\in E_A\big\}
\le \sup\big\{|T_\tau'(\hat\pi_2(\tau))|:\tau\in E_A\big\}
\le M.
\eeq
Having this and inserting \eqref{2vl6} to \eqref{1bf4}, we get
\beq\label{3bf4}
\nu_\om(B(z,r))\ge (DM^h)^{-1}r^h.
\eeq
To prove an appropriate inequality in the opposite direction, let
$$
\begin{aligned}
\cF(z,r):=
\bigg\{\tau\in &E_A^*(-\infty,-1):T_{\tau\om}^{|\tau|}(Y_{\tau\om})\cap B(z,r/2)\ne\es, \\
  &\diam\(T_{\tau\om}^{|\tau|}(Y_{\tau\om})\)\le r/2 
\  \and \
\diam\Big(T_{\tau|_{-(|\tau|-1)}^{-1}\om}^{|\tau|}\(Y_{\tau|_{-(|\tau|-1)}^{-1}\om}\)\Big)
> r/2 \bigg\}.
\end{aligned}
$$
Thus $\cF(z,r)$ consists of mutually incomparable elements of $E_A^*(-\infty,-1)$, so using \eqref{2bf4} and (BDP3), we get for every $\tau\in \cF(z,r)$ with $n:=|\tau|$, that
$$
\begin{aligned}
\diam\(T_{\tau\om}^n(Y_{\tau\om})\)
&=\diam\Big(T_{\tau|_{-(n-1)}^{-1}\om}^{n-1}\(T_{\tau\om}(Y_{\tau\om})\)\Big)
\ge K^{-1}\Big\|\Big(T_{\tau|_{-(n-1)}^{-1}\om}^{n-1}\Big)'\Big\|_\infty
     \diam\(T_{\tau\om}(Y_{\tau\om})\)\\
&\ge K^{-2}\Big\|\Big(T_{\tau|_{-(n-1)}^{-1}\om}^{n-1}\Big)'\Big\|_\infty
     \big\|T_{\tau\om}'\big\|_\infty\diam(Y_{\tau\om}) \ge 2K^{-2}M^{-1}R\Big\|\Big(T_{\tau|_{-(n-1)}^{-1}\om}^{n-1}\Big)'\Big\|_\infty \\
&\ge 2K^{-3}M^{-1}R\,\diam(Y)^{-1}\diam\Big(T_{\tau|_{-(n-1)}^{-1}\om}^{n-1}\(T_{\tau\om}
     \(Y_{\tau|_{-(n-1)}^{-1}\om}\)\Big) \\
&\ge K^{-3}M^{-1}R\,\diam(Y)^{-1}r
\end{aligned}
$$
Thus Lemma~\ref{l1ct3A} applies with the radius equal to $r/2$, given that
$
\#\cF(z,r)\le \Ga_\ga,
$
where 
$
\ga:=\max\{1,2K^3MR^{-1}\diam(Y)\}.
$
Since also
$
\hat\pi_2^{-1}(B(z,r))\sbt \bu_{\tau\in\cF(z,r)}[\tau\om],
$
we therefore get
$$
\begin{aligned}
\nu_\om(B(z,r))
&\le \ov\mu_h^\om\circ\hat\pi_2^{-1}(\bu_{\tau\in\cF(z,r)}[\tau\om])
 \le \sum_{\tau\in\cF(z,r)}\ov\mu_h^\om\circ\hat\pi_2^{-1}([\tau\om]) \\
&\le K^h\sum_{\tau\in\cF(z,r)}\diam^h\(T_{\tau\om}^{|\tau|}(Y_{\tau\om})\) \le 2^hK^h\#_\Ga r^h
\end{aligned}
$$
along with \eqref{3bf4} this shows that $\nu_\om$ is Ahlfors regular with exponent $h=B_T$.
\epf

\sp{\bf Proof of Theorem~\ref{t1bf2}:} Fix $t>B_T$ arbitrary; then $\P(t)<0$, so for every integer $n\ge 1$ large and $\om\in E_A^+$, we have 
$$
\sum_{\tau\in E_A^*(-n,-1)\atop A_{\tau_{-1}\om_0}=1}\big\|\(T_{\tau\om}^n\)'\big\|_\infty^t
\le \exp\lt(\frac12\P(t)n\rt).
$$
Thus by (BDP2),
\beq\label{1bf5}
\sum_{\tau\in E_A^*(-n,-1)\atop A_{\tau_{-1}\om_0}=1}\diam^t\(T_{\tau\om}^n(Y_{\tau\om})\)
\le K^t\exp\lt(\frac12\P(t)n\rt).
\eeq
$\P(t)<0$, and $\big\{T_{\tau\om}^n(Y_{\tau\om}):\tau\in E_A^*(-n,-1),\, A_{\tau_{-1}\om_0}=1\big\}$ is a cover of $J_\om$, and diameters of its sets converge to zero; so, from \eqref{1bf5},  $\H_t(J_\om)=0$. Thus 
$\HD(J_\om)\le t$ and
$$
\HD(J_\om)\le B_T. 
$$
For the opposite inequality, fix $0\le t<B_T$. Then $\P(t)>0$ and from Theorem~\ref{t2_2015_11_17}, $\P_F(t)>0$ for a finite set $F\sbt E$ s.t $A|_{F\times F}$ is irreducible.  Theorem~\ref{t1bf3} gives $\HD(J_\om(F))>t$, $\forall \om\in E_A^+$. But $J_\om(F)\sbt J_\om$, hence $\HD(J_\om)\ge t$. As $t<B_T$ is arbitrary, we get
$$\HD(J_\om)\ge B_T$$
$\hfill\square$

\section{General Skew Products over Countable--to--1 Endomorphisms.}\label{generalskew}
 We want to enlarge the class of endomorphisms with exact dimensionality of conditional measures on fibers. For general thermodynamic formalism of endomorphisms related to our approach, one can see \cite{Ru}, \cite{M-stable}, \cite{M-DCDS}, \cite{M-ETDS11}, \cite{M-MZ}, \cite{MS}, \cite{MU-JFA}, etc.   Our results on exact dimensionality
  of conditional measures in fibers extend a result on exact dimensionality of conditional measures on stable manifolds of hyperbolic endomorphisms from \cite{M-stable}.
We will apply the results obtained in previous sections to skew products over countable-to-1 endomorphisms. This includes EMR maps, continued fractions transformation, etc.

First, we prove a result about skew products whose base transformations are modeled by 1-sided shifts on a countable alphabet.
Assume we have a skew product $F: X \times Y \to X \times Y$, where $X$ and $Y$ are complete bounded metric spaces, $Y \subset \R^d$ for some $d \ge 1$, and 
$$
F(x, y) = (f(x), g(x, y)),
$$ 
where the map 
$$
Y\ni y \longmapsto g(x, y)
$$
is injective and continuous for every $y \in Y$.\ 
Denote the map $Y\ni y \mapsto g(x, y)$ also by $g_x(y)$. 
Assume $f:X \to X$ is at most countable-to-1, and its dynamics is modeled by a 1-sided Markov shift on a countable alphabet $E$ with the matrix $A$ finitely irreducible, i.e there exists a surjective H\"older continuous map, called \textit{coding}, 
$$
p: E_A^+ \to  X \  \  \text{such that} \ \ p\circ \sigma = f\circ p
$$
Assume conditions (a)--(g) from Section~\ref{CSPSE} are satisfied for  $T_\om:Y_\om\to Y_{\sg_\om}$ given by $T_\omega := g_{p(\omega)},  \  \om \in E_A^+$. Then we call $F: X \times Y \to X \times Y$ a \textit{generalized conformal skew product Smale endomorphism. }

Given the skew product $F$ as above, we can also form a  skew product endomorphism in the following way:
define for every $\om \in E_A^+$, the fiber map $\hat F_\om: Y \to Y$ by 
$$
\hat F_\om(y) = g(p(\om),y).
$$
The system $(\hat Y, \hat F)$ is called \textit{the symbolic lift} of $F$.
If $\hat Y = E_A^+ \times Y$, we obtain a conformal skew product Smale endomorphism $\hat F: \hat Y \to \hat Y$ given by
\begin{equation}\label{skew-general}
\hat F(\omega, y) = (\sigma(\om), \hat F_\om(y))
\end{equation}

As in the beginning of Section~\ref{con}, we study the structure of fibers $J_\om, \ \om \in E_A^+$ and later of the sets $J_x$, $x\in X$. From definition, $J_\om  = \hat\pi_2([\om])$ and it is the set of points of type 
$$
\bi_{n \ge 1}\ov{\hat F_{\tau_{-1}\om}\circ \hat F_{\tau_{-2}\tau_{-1}}\om\circ\ldots \circ \hat F_{\tau_{-n}\ldots\tau_{-1}\om}(Y)}.
$$
Let us call $n$-\textit{prehistory} of the point $x$ with respect to the system $(f, X)$, any finite sequence of points in $X$:
$
(x, x_{-1}, x_{-2}, \ldots, x_{-n})\in X^{n+1},
$
where \
$
f(x_{-1}) = x,\, f(x_{-2}) = x_{-1}, \ldots, f(x_{-n}) = x_{-n+1}.
$ \ 
Call a \textit{complete prehistory} (or simply a \textit{prehistory}) of $x$ with respect to  $(f, X)$, any infinite sequence of consecutive preimages in $X$, i.e. 
$
\hat x = (x, x_{-1}, x_{-2}, \ldots),
$
where  
$
f(x_{-i}) = x_{-i+1},
$
 $i \ge -1$.
The space of complete prehistories is denoted by $\hat X$ and is called the \textit{natural extension} (or \textit{inverse limit}) of  $(f, X)$. We have a bijection $\hat f: \hat X \to \hat X$,
$$
\hat f(\hat x) = (f(x), x, x_{-1}, \ldots).
$$
In this paper, we use the terms inverse limit and natural extension interchangeably, without having necessarily a fixed invariant measure defined on the space $X$. 

\fr We consider on $\hat X$ the canonical metric, which induces the topology equivalent to the one inherited from the product topology on $X^{\N}$. Then $\hat f$ becomes a homeomorphism. For more on dynamics of endomorphisms and inverse limits, one can see for eg. \cite{DKS}, \cite{Ru}, \cite{M-DCDS}, \cite{FM}, \cite{M-stable}, \cite{MS}, \cite{M-MZ}, \cite{M-JSP11}, \cite{M-ETDS11}, \cite{M-Mon}.

In the above notation, we have  
$
f(p(\tau_{-1}\om)) = p(\om)=x,
$
and for all the prehistories of $x$, $\hat x = (x, x_{-1}, x_{-2}, \ldots) \in \hat X$, consider the set $J_x$ of points of type 
$$\bigcap_{n \ge 1} \ov{g_{x_{-1}} \circ g_{x_{-2}}\circ \ldots \circ g_{x_{-n}}(Y)}$$

\fr Notice that, if $\hat \eta = (\eta_0, \eta_1, \ldots)$ is another sequence in $E_A^+$ such that $p(\hat\eta) = x$, then for any $\eta_{-1}$ so that $\eta_{-1}\hat\eta \in E_A^+$, we have $p(\eta_{-1}\hat \eta) = x_{-1}'$ where $x_{-1}'$ is some 1-preimage (i.e preimage of order 1) of $x$. 
Hence from the definitions and the discussion above, we see that 
\begin{equation}\label{J_x}
J_x = \mathop{\bigcup}\limits_{\omega\in E_A^+, p(\omega) = x} J_\omega
\end{equation}

Let us denote the respective fibered limit sets for $T$ and $F$ by: 
\begin{equation}\label{J(X)}
J = \bigcup_{\omega \in E_A^+} \{\omega\}\times J_\omega \subset E_A^+ \times Y  \  \ \text{and} \  \  J(X):= \bigcup_{x \in X}\{x\}\times J_x \subset X \times Y
\end{equation}
So
$
\hat F(J)=J$
  and $ 
F(J(X))=J(X).
$
The H\"older continuous projection $p_J: J \to J(X)$ is 
$$
p_J(\omega, y) = (p(\omega), y),
$$
we obtain 
$
F\circ p_J = p_J \circ \hat F$. \ 
In the sequel, $\hat\pi_2:E_A\to Y$ and $\hat\pi:E_A\to E_A^+\times Y$ are the maps defined in Section~\ref{con} and,
$$
\hat \pi(\tau) = (\tau|_0^\infty, \hat\pi_2(\tau)).
$$
Now, it is important to know if enough points $x \in X$ have unique coding sequences in $E_A^+$. 

\begin{dfn}\label{phi-inj}
Let $F: X \times Y \to X \times Y$ be a generalized conformal skew product Smale endomorphism. Let $\mu$ be a Borel probability measure $X$. We then say that the coding $p:E_A^+ \to X$ is $\mu$-injective, if there exists a $\mu$-measurable set $G \subset X$ with $\mu(G) = 1$ such that for every point $x \in G$, the set $p^{-1}(x)$ is a singleton in $E_A^+$. 

\fr Denote such a set $G$ by $G_\mu$ and for $x\in G_\mu$ the only element of $p^{-1}(x)$ by $\om(x)$.
\end{dfn}

\begin{prop}\label{phi-i}
If the coding $p:E_A^+ \to X$ is $\mu$-injective, then for every $x\in G_\mu$, we have
$$
J_x = J_{\omega(x)}.
$$
\end{prop}

\begin{proof}
Take $x\in G_\mu$, and let $x_{-1}\in X$ be an $f$-preimage of $x$, i.e $f(x_{-1}) = x$. Since $p:E_A^+ \to X$ is surjective, there exists  $\eta \in E_A^+$ such that $p(\eta) = x_{-1}$. But this implies that 
$
f(x_{-1}) = f\circ p(\eta) = p\circ \sigma(\eta) = x.
$
Then, from the uniqueness of the coding sequence for $x$, it follows that $\sigma(\eta) = \omega(x)$, whence $x_{-1} =p(\omega_{-1}\omega(x))$, for some $\omega_{-1}\in E$. Since
$
J_x=\bigcap_{n \ge 1}\ov{g_{x_{-1}} \circ g_{x_{-2}}\circ \ldots \circ g_{x_{-n}}(Y)},
$
it follows that $J_x = J_{\omega(x)}$. 

\end{proof}

In the sequel we work only with $\mu$-injective codings, and the measure $\mu$ will be clear from the context. Also given a metric space $X$ with a coding $p:E_A^+ \to X$, and a potential $\phi:X \to \mathbb R$, we say that $\phi$ is \textit{ H\"older continuous} if $\phi\circ p$ is  H\"older continuous.

 Now consider a potential $\phi: J(X) \to \R$ such that the potential 
$$
\widehat\phi: = \phi\circ p_J\circ \hat \pi: E_A \to \R
$$
is  H\"older continuous and summable. For example, $\widehat\phi$ is  H\"older continuous if $\phi: J(X) \to \R$ is itself  H\"older continuous. This case will be quite frequent in examples below. Define now
\beq\label{1_2017_01_13}
\mu_\phi := \mu_{\widehat\phi} \circ (p_J\circ \hat \pi)^{-1},
\eeq
and call it the equilibrium measure of $\phi$ on $J(X)$ with respect to the skew product $F$. 

\sp Now, let us consider the partition $\xi'$ of $J(X)$ into the fiber sets $\{x\} \times J_x, \ x\in X$, and the conditional measures 
$\mu^x_\phi$ associated to $\mu_\phi$ with respect to the measurable partition $\xi'$ (see \cite{Ro}). Recall that for each $\omega \in E_A^+$, we have 
$
\hat\pi_2([\omega]) = J_\omega.
$

\

Denote by $p_1:X\times Y\to X$  the canonical projection onto the first coordinate, i.e. 
$$
p_1(x,y)=x.
$$

\begin{thm}\label{sing}
Let $F: X \times Y \to X \times Y$ be a generalized conformal skew product Smale endomorphism. Let $\phi:J(X) \to \R$ be a potential such that $\widehat\phi= \phi \circ p_J\circ \hat \pi: E_A \to \R$ is a  H\"older continuous summable potential on $E_A$. 
Assume that the coding $p:E_A^+ \to X$ is $\mu_\phi\circ p_1^{-1}$--injective, and  denote the corresponding set $G_{\mu_\phi}\sbt X$ by $G_\phi$. Then:

\sp\begin{enumerate}
\item $J_x = J_{\omega(x)}$ for every $x\in G_\phi$.

\sp\item With $\bar\mu^\omega_\psi$, $\omega \in E_A^+$ the conditional measures of $\mu_{\widehat\phi}$, we have for $\mu_\phi\circ p_1^{-1}$--a.e.  $x\in G_\phi$,
$$
\mu_\phi^x = \bar\mu_{\widehat\phi}^{\omega(x)}\circ (p_J\circ\hat\pi)^{-1}, 
$$
or equivalently, if  $\mu_\phi^x$ and $\bar\mu_{\widehat\phi}^{\omega(x)}$ are viewed   on $J_x$ and $E_A^-$, 
$
\mu_\phi^x = \bar\mu^{\omega(x)}_{\widehat\phi}\circ\hat\pi_2^{-1} 
$.
\end{enumerate}
\end{thm}

\begin{proof}
Part (1) is just Proposition~\ref{phi-i}. We thus deal with part (2) only. By the definition of conditional measures, we have for every $\mu_\phi$--integrable function $H:J(X)\to\R$ that
\beq\label{1_2017_02_08}
\int_{J(X)} H d\mu_\phi 
= \int_{E_A}H\circ p_J\circ \hat \pi d\mu_{\widehat\phi}
= \int_{E_A^+} \int_{[\omega]}H\circ p_J\circ \hat \pi\, d\bar\mu_{\widehat\phi}^\omega \ d\mu_{\widehat\phi}\circ \pi_1^{-1}(\omega)
\eeq
and
\beq\label{2_2017_02_08}
\int_{J(X)}H d\mu_\phi 
=\int_X \int_{\{x\}\times J_x} H\, d\mu_\phi^x\, d\mu_\phi\circ p_1^{-1}(x)
\eeq
But from the definitions of various projections: 
\beq\label{1_2017_02_23}
\mu_\phi\circ p_1^{-1}
= \mu_{\widehat\phi}\circ (p_J\circ \hat \pi)^{-1}\circ p_1^{-1}
= \mu_{\widehat\phi}\circ (p_1\circ p_J\circ \hat \pi)^{-1}            
= \mu_{\widehat\phi}\circ (p\circ\pi_1)^{-1}
= \mu_{\widehat\phi}\circ \pi_1^{-1}\circ p^{-1}.
\eeq
Therefore, remembering also that $\mu_\phi\circ p_1^{-1}(G_\phi)=1$, we get that
\begin{equation}\label{phipsi}
\begin{aligned}
\int_{E_A^+} \int_{[\omega]}H\circ p_J \circ \hat \pi\, & d\bar\mu_{\widehat\phi}^\omega \ d\mu_{\widehat\phi}\circ \pi_1^{-1}(\omega)=\int_{E_A^+} \int_{\{p(\om)\}\times J_{p(\om)}} H d\bar\mu_{\widehat\phi}^\omega\circ (p_J\circ \hat \pi)^{-1}\, d\mu_{\widehat\phi}\circ \pi_1^{-1}(\omega) \\
&=\int_{G_\phi}\int_{\{x\}\times J_x} H d\bar\mu_{\widehat\phi}^{\omega(x)}\circ (p_J\circ \hat \pi)^{-1}\, d\mu_{\widehat\phi}\circ \pi_1^{-1}\circ p^{-1}(x) \\
&=\int_{G_\phi}\int_{\{x\}\times J_x} H d\bar\mu_{\widehat\phi}^{\omega(x)}\circ (p_J\circ \hat \pi)^{-1}\,  d\mu_\phi\circ p_1^{-1}(x).
\end{aligned}
\end{equation}
Hence this, together with \eqref{1_2017_02_08} and \eqref{2_2017_02_08}, gives
$$
\int_{G_\phi}\int_{\{x\}\times J_x} H\, d\mu_\phi^x\, d\mu_\phi\circ p_1^{-1}(x)
=\int_{G_\phi}\int_{\{x\}\times J_x} H d\bar\mu_{\widehat\phi}^{\omega(x)}\circ (p_J\circ \hat \pi)^{-1}\,  d\mu_\phi\circ p_1^{-1}(x).
$$
Thus, the uniqueness of the system of Rokhlin's canonical conditional measures yields
$
\mu_\phi^x=\bar\mu_{\widehat\phi}^{\omega(x)}\circ (p_J\circ\hat\pi)^{-1}
$
for $\mu_\phi\circ p_1^{-1}$--a.e. $x\in G_\phi$. This means that the first part of (2) is established. But $
p_J\circ \hat \pi=(p\circ\pi_1)\times \hat\pi_2$, and thus $
p_J\circ \hat \pi|_{[\om(x)]}=\{x\}\times\hat\pi_2|_{[\om(x)]}$. 

\end{proof}

 Define the \textit{Lyapunov exponent} for an $F$-invariant measure $\mu$ on  $J(X) = \bigcup\limits_{x\in X} \{x\}\times J_x$ by: 
$$
\chi_\mu(F) = - \int_{J(X)} \log|g_x'(y)| \,  d\mu(x,y). 
$$

\

\fr In conclusion, from Theorem~\ref{sing}, Theorem~\ref{t1vl4}, and definition \eqref{1_2017_01_13}, we obtain the following result for skew product endomorphisms over countable-to-1 maps $f:X \to X$:

\begin{thm}\label{codedskews}
Let $F: X \times Y \to X \times Y$ a generalized conformal skew product Smale endomorphism. Let $\phi:J(X) \to \R$ be a potential such that 
$$
\hat\phi:= \phi \circ p_J\circ \hat \pi: E_A \to \R
$$ 
is  H\"older continuous summable. 
Assume the coding $p:E_A^+ \to X$ is $\mu_\phi\circ p_1^{-1}$--injective.

\fr Then for $\mu_\phi\circ p_1^{-1}$-a.e $x \in X$, the conditional measure $\mu_\phi^x$ is exact dimensional on $J_x$, and
$$
\lim_{r\to 0} \frac{\log \mu_\phi^x(B(y, r))}{\log r} 
= \frac{\h_{\mu_\phi}(F)}{\chi_{\mu_\phi}(F)}
=\HD(\mu_\phi^x),
$$
for $\mu_\phi^x$-a.e $y \in J_x$; hence, equivalently, for $\mu_\phi$-a.e $(x, y) \in J(X)$.
\end{thm}

As an immediate consequence of this theorem, we get the following.

\begin{cor}\label{codedskews_A}
Let $F: X \times Y \to X \times Y$ a generalized conformal skew product Smale endomorphism. Let $\phi:J(X) \to \R$ be a  H\"older continuous potential such that 
$$
\sum_{e \in E} \exp(\sup(\phi|_{\pi([e])\times Y}))< \infty.
$$ 
Assume that the coding $p:E_A^+ \to X$ is $\mu_\phi\circ p_1^{-1}$--injective.
Then, for $\mu_\phi\circ p_1^{-1}$-a.e $x \in X$, the conditional measure $\mu_\phi^x$ is exact dimensional on $J_x$, and for $\mu_\phi^x$-a.e $y \in J_x$,
$$
\lim_{r\to 0} \frac{\log \mu_\phi^x(B(y, r))}{\log r} 
= \frac{\h_{\mu_\phi}(F)}{\chi_{\mu_\phi}(F)}
=\HD(\mu_\phi^x)
$$
\end{cor}

\

By using Theorem \ref{codedskews}, we will  prove exact dimensionality of conditional measures of equilibrium states on fibers for many types of skew products.

First, we prove a  result about \textit{global} exact dimensionality of measures on  $J(X)$. 

\begin{thm}\label{globalexact}
Let $F: X \times Y \to X \times Y$ a generalized conformal skew product Smale endomorphism. Assume that $X\sbt\R^d$ with some integer $d\ge 1$. Let $\mu$ be a Borel probability $F$--invariant measure on $J(X)$, and $(\mu^x)_{x\in X}$ be the Rokhlin's canonical sytem of conditional measures of $\mu$, with respect to the partition $\(\{x\}\times J_x\)_{x\in X}$. Assume that:
\begin{itemize}
\item[(a)]  There exists $\a>0$ such that for $\mu\circ p_1^{-1}$-a.e $x \in X$, the conditional measure $\mu^x$ is exact dimensional and $\HD(\mu^x)=\a$, 

\item[(b)] The measure $\mu\circ p_1^{-1}$ is exact dimensional on $X$. 
\end{itemize}

 \fr Then, the measure $\mu$ is exact dimensional on $J(X)$, and for $\mu$-a.e $(x,y) \in J(X)$,
$$
\HD(\mu)=\lim\limits_{r\to 0} \frac{\log \mu(B((x, y), r))}{\log r} = \alpha + \HD(\mu\circ p_1^{-1}).
$$
\end{thm}

\begin{proof}
Denote the canonical projection to first coordinate by  $p_1: X\times Y \to X$. Let then
$
\nu:= \mu\circ p_1^{-1}.
$
Denote the Hausdorff dimension $\HD(\nu)$ by $\gamma$. From the exact dimensionality of the conditional measures of $\mu$, we know that for $\nu$--a.e $x \in X$ and for $\mu^x$--a.e $y \in Y$, 
$$
\lim\limits_{r\to 0} \frac{\log \mu^x(B(y, r))}{\log r} = \alpha.
$$ 
Then for any $\e\in(0,\a)$ and any integer $n\ge 1 $, consider the following Borel set in $X \times Y$:
$$
A(n, \e) 
:= \lt\{z=(x,y) \in X \times Y: \ \alpha-\e< \frac{\log \mu^x(B(y, r))}{\log r} < \alpha + \e \  \ {\rm for \ all } \ \  r \in (0,1/n)\rt\}.
$$ 
From definition it is clear that $A(n, \e) \subset A(n+1, \e)$ for all  $n \ge 1$. Moreover, setting
$$
X_Y':= \bigcap_{\e>0}\bigcup_{n=1}^\infty A(n, \e),
$$
it follows from the exact dimensionality of almost all the conditional measures of $\mu$ and from the equality of their pointwise dimensions, that 
$
\mu(X_Y') = 1.
$
For $\e>0$ and $n\ge 1$, consider also the following Borel subset of $X$:
$$
D(n , \e)
:= \lt\{x \in X: \ \gamma-\e < \frac{\log \nu(B(x, r))}{\log r} < \ga +\e  \  \ {\rm for \ all } \ \  r \in (0,1/n)\rt\}.
$$
We know that $D(n, \e) \subset D(n+1, \e)$ for all $n \ge 1$, and from the exact dimensionality of $\nu$, we obtain that for every $\e>0$, we have
$
\nu\lt(\bigcup_{n=1}^\infty D(n, \e)\rt) = 1.
$
For $\e>0$ and an integer $n \ge 1$, let us denote now 
$$
E(n, \e):= A(n, \e) \cap p_1^{-1}(D(n, \e)).
$$
Clearly from above, we have that for any $\e >0$,
\beq\label{1_2016_11_04} 
\lim_{n\to\infty}\mu(E(n, \e))=1.
\eeq
Using the definition of conditional measures similarly as in \cite{M-stable}, \cite{M-JSP11}, and the definition of $A(n, \e)$ and $D(n, \e)$, we have that, for any $z \in E(n, \e), \ x = \pi_1(z)$ and any $n\ge 1, \e>0, 0 < r <1/n$,  
\begin{equation}\label{upmu}
\begin{aligned}
\mu(E(n, \e) \cap B(z, r)) &= \int_{D(n, \e) \cap B(x, r)} \mu^y\big(B(z, r) \cap (\{y\}\times Y) \cap A(n, \e)\big) \ d\nu(y) \\
&\le \int_{D(n, \e)\cap B(x, r)} r^{\alpha-\e} \ d\nu(y) =
 r^{\alpha - \e} \nu(D(n, \e) \cap B(x, r)) \\
 &\le r^{\alpha +\ga - 2\e}
\end{aligned}
\end{equation}
Since $\mu(E(n,\e))>0$ for all $n\ge 1$ large enough, it follows from Borel Density Lemma - Lebesgue Density Theorem that, for $\mu$-a.e $z \in E(n, \e)$, we have that
$$
\lim_{r\to 0} \frac{\mu(B(z, r)\cap E(n, \e))}{\mu(B(z, r))} = 1.
$$
Thus for any $\theta >1$ arbitrary, there exists a subset $E(n, \e, \theta)$ of $E(n, \e)$, such that 
$$
\mu(E(n, \e, \theta) ) = \mu(E(n, \e)),
$$
and $\forall z \in E(n, \e, \theta)$, $\exists r(z, \theta) >0$ so that for $0 < r < \inf\{r(z, \theta),1/n\}, $ we have from \ref{upmu}:
$$\mu(B(z, r)) \le \theta \mu(E(n, \e) \cap B(z, r)) \le \theta \cdot r^{\alpha + \ga -2\e}$$
Thus for any point $z \in E(n, \e, \theta)$, we obtain $$\lim_{r\to 0} \frac{\log \mu(B(z, r))}{\log r} \ge \alpha + \ga -2\e$$
Now, since $\mu(E(n, \e, \theta)) = \mu(E(n, \e))$, it follows from \eqref{1_2016_11_04} that $\mu(\mathop{\cup}\limits_n E(n, \e, \theta)) = 1$. Hence  
$$
\mu\lt(\bi_{\e>0}\bi_{\theta>1}
\bu_{n=1}^\infty E(n, \e, \theta)\rt) = 1,
$$
and for every $z \in  \bigcap\limits_{\e>0}\bigcap\limits_{\theta>1}\bigcup\limits_n E(n, \e, \theta) $, we have the inequality
$$
\lim_{r\to 0} \frac{\log \mu(B(z, r))}{\log r} \ge \alpha + \ga.
$$
Conversely, from the exact dimensionality of $\nu$ and of the conditional measures of $\mu$, and  by taking $x = \pi_1(z)$, we have that for all $r \in (0,1/n)$, 
\begin{equation}\label{lowmu}
\begin{aligned}
\mu(B(z, r) \cap E(n, \e)) &= \int_{D(n, \e) \cap B(x, r)} \mu^y\(B(z, r) \cap A(n, \e)\cap \{y\}\times Y\) \, d\nu(y)  \ge r^{\alpha + \ga +2\e}
\end{aligned}
\end{equation}
Thus, 
$$
\mu(B(z, r)) \ge \mu(B(z, r) \cap E(n, \e)) \ge r^{\alpha + \ga +2\e},
$$
for all $z \in E(n, \e)$ and $r \in (0,1/n)$. Making use of \eqref{1_2016_11_04}, we deduce that $\mu$ is exact dimensional, and thus for $\mu$-a.e $z \in X\times Y$ we obtain the conclusion: 
$$
\mathop{\lim}\limits_{r\to 0} \frac{\log \mu(B(z, r))}{\log r}
 = \alpha + \ga$$
\end{proof}

\section{Skew products over EMR-endomorphisms.}\label{S-EMR}

We now consider EMR (expanding Markov-R\'enyi) maps on the interval, and we construct skew product endomorphisms over these maps which contract in fibers. This EMR class contains important examples of endomorphisms coded by a shift space with countable alphabet, such as the continued fractions Gauss transformation. The Manneville-Pomeau map, having an indifferent fixed point (\textit{parabolic point}), is not EMR, but one can associate to it a countable uniformly hyperbolic system by inducing via the first return map. Let us first give the definition of EMR maps from \cite{PoW}.

\begin{dfn}\label{EMR}
Let I be a closed bounded interval in $\R$, and assume that 
$
I = \bu_{n=0}^\infty I_n,
$
where $I_n$, $n \ge 0$, are closed intervals with mutually disjoint interiors. A map 
$
f:I \to I
$
is called EMR if:
\begin{itemize}
\item[(a)] For every $n\ge 0$, $f|_{\Int(I_n)}$ is a $\mathcal C^2$ map.

\,

\item[(b)] There exists an iterate of $f$ which is uniformly expanding, i.e there exist $\l>1$ and a positive integer $m$, such  that 
$$
|(f^m)'(x)| \ge \l, 
$$
for all 
$x
\in\bi_{k=0}^m f^{-k}\lt(\bu_{n \ge 0} \Int(I_n)\rt).
$
\item[(c)] The map $f$ is full Markov, i.e 
$
f(I_n)=I,
$
for every integer $n \ge 0$.

\,

\item[(d)] The map $f$ satisfies R\'enyi's condition, i.e. 
$$
\sup\lt\{\frac{|f''(x)|}{|f'(y)|\cdot |f'(z)|}:n\in\N \  \  {\rm and}  \  \  x, y, z \in I_n \rt\} < +\infty.
$$
\end{itemize}
\end{dfn} 

\fr For an EMR map $f:I\to I$, there is a coding with a shift space on countably many symbols
$$
p: \N^\N \lra I,
$$
uniquely determined by the condition that
$$
\{p(\om)\} = \bi_{n=0}^\infty f^{-n}(I_{\om_n}).
$$
Every point $x$ which never hits  the boundary of any interval $I_n$ under an iterate of $f$, has a unique such coding, i.e there exists a unique $\om \in \N^\N$ with $p(\om) = x$. Thus, the $p:E_A^+ \to X$ is injective outside a countable set.

\sp Two significant classes of examples used in the sequel are: the continued fractions map as a class in itself, and a a class derived from Manneville--Pomeau maps. The continued fractions (Gauss) map is $G:[0, 1] \to [0, 1]$, 
$$
G(x):=\lt\{\frac 1x\rt\}= \frac 1x - \lt[\frac 1x\rt]  \  {\rm for} \  x \ne 0, \ \text{and} \ f_1(0) = 0$$
A Manneville--Pomeau maps $f_{\a,a}:[0, 1] \to [0, 1]$ are defined by:
$$
f_\a(x) = x + x^{1+\alpha} \  (\text{mod} \ 1),
$$
where $\a>0$.
Notice that  $f_\a$ has an indifferent fixed point at $0$, i.e $f_\a(0)=0$ and $f_\a'(0)=1$. So $f_2$ is not an EMR map. It was shown, and it is easy to see, that some induced (first return) map of $f_\a$ is EMR. Indeed, let $c$ be the unique solution on $[0,1]$ to the equation
$$
c_\a+c_\a^\alpha=1.
$$
Obviously, both restrictions
$$
f_\a|_{[0, c_\a]} 
\  \  \  {\rm and}  \   \  \
f_\a|_{[c_\a,1]}
$$
are injective and 
$
f_\a([0, c])= [0,1]=f_\a([c,1]).
$
The first return (induced) map of $f_\a$ from $[c_\a, 1]$ onto $[c_\a, 1]$ is expanding, i.e. it satisfies condition (b) of Definition~\ref{EMR}. The reason is that $f'(x)>1$ for every $x\in(0,1]$ and if $x\in(c,1)$, then it will take sufficiently long time for $x$ to enter
$[c,1]$ again under forward iteration of $f_\a$, that the derivative of this iterate evaluated at $x$ will be larger than some constant larger that $1$. In fact, denoting the interval $[c_\a,1]$ by $I_0$, it is very easy to picture the first return map
$$
f_{\a,I_0}:I_0\to I_0
$$
Indeed, define a decreasing sequence $(a_n)_{n=0}^\infty$ inductively as follows. $a_0:=1$ and if $a_n\in I_0$, $n\ge 0$, has bee defined, then $a_{n+1}$ is the defined to be the only point in $I_0$ such that
$$
f(a_{n+1})=a_n.
$$
For every $n\ge 1$ let  
$
I_n:= [a_n, a_{n-1}].
$
Then  
$
f_2^n(I_{n+1}) = I_0
$.
 for all $n >0$, so the induced map (first return time) to $I_0$ is:
\begin{equation}\label{indf2}
f_{2, I_0}(x) = f_2^n(x), \ x \in I_{n+1}, n \ge 0
\end{equation}

The first statement of the following proposition follows from \cite{Sch}, while the second one follows from \cite{T}.
\begin{prop}\label{TS} 
Both the Gauss map $G:[0, 1] \to [0, 1]$ and the induced Manneville--Pomeau maps $f_{\a,I_0}:I_0\to I_0$, are EMR. 
\end{prop}

Consider now a general EMR map $f:I \to I$, and a skew product $F: I \times Y \to I \times Y$, where $Y\subset \R^d$ is a bounded open set, with $F(x, y) = (f(x), g(x, y))$. Recall that the symbolic lift of $F$ is the map $\hat F: \N^\N \times Y \lra \N^\N \times Y$, given by the formula: 
$$
\hat F(\omega, y)= (\sigma(\omega), g(p(\omega), y))
$$
forall $(\omega, y) \in \N^\N\times Y$. If the symbolic lift $\hat F$ is a H\"older conformal skew product Smale endomorphism, then we say by extension that $F$ is a \textit{H\"older conformal skew product endomorphism over $f$}.

\

 Recall now the observation after Definition \ref{EMR}, that the coding $\pi$ of an EMR map is injective outside a countable set; and,  from (\ref{J(X)}), the fibered limit set of $F$ is 
$$
J(I) = \bu_{x \in I} \{x\} \times J_x.
$$ Let $\phi:J(I)\to\R$ be a  H\"older continuous summable potential on $J(I)$, and let $\mu_\phi$ be its equilibrium measure defined in \eqref{1_2017_01_13}.  Then $p$ can be easily shown  to be $\mu_\phi\circ p_1^{-1}$-injective (as $\mu_\phi$ is invariant, ergodic and has full topological support).
So, as an immediate consequence of from Theorem~\ref{codedskews}, we get the  following.

\begin{thm}\label{EMRskews}(Exact dimensionality of conditional measures for EMR maps).
Let $f:I \to I$ an EMR map. Let $Y \subset \R^d$ be an open bounded set, and let $F:I \times Y \to I\times Y$ be a H\"older conformal skew product endomorphism over $f$.  
Let $\phi: J(I) \to \R$ such that 
$$
\hat\phi:=\phi\circ p_J\circ\hat\pi:\N^{\Z}\lra\R
$$
is H\"older continuous summable, where summability can be expressed as 
$$
\sum_{e \in \N} \exp\(\sup(\phi|_{p([e])\times Y})\)  < +\infty
$$
Then, with $\mu_\phi$ being the equilibrium measure of $\phi$ (see also \eqref{1_2017_01_13}), we have for $\mu_\phi\circ p_1^{-1}$--a.e $x \in I$, that the conditional measure $\mu_\phi^x$ is exact dimensional on $J_x$ and 
$$
\lim_{r\to 0} \frac{\log \mu_\phi^x(B(y, r))}{\log r} 
= \frac{h_{\mu_\phi}(F)}{\chi_{\mu_\phi}(F)}
=\HD\(\mu_\phi^x\)
$$
for $\mu_\phi^x$-a.e $y \in J_x$; equivalently for $\mu_\phi$-a.e $(x, y) \in J(I)$.
\end{thm}

\fr As an immediate consequence of this theorem, we get the following.

\begin{cor}\label{EMRskewsB}(Exact dimensionality of conditional measures for EMR maps, II).
Let $f:I \to I$ an EMR map. Let $Y \subset \R^d$ be an open bounded set, and let $F:I \times Y \to I\times Y$ be a H\"older conformal skew product endomorphism over $f$.  
Let $\phi:I\to\R$ be a potential. Set 
$$
\psi:=\phi\circ p_1:I\times Y\lra\R.
$$
Assume that 
$
\widehat\psi=\phi\circ p_1\circ p_J\circ\hat\pi:\N^{\Z}\lra\R
$
is H\"older continuous summable, where summability can be expressed as 
$
\sum_{e \in \N} \exp\(\sup(\phi|_{I_n})\) < +\infty.
$
Then, 
\begin{itemize}
\item[(a)] If $\mu_\psi$ is the equilibrium measure of $\psi$, and $\mu_\phi=\mu_\psi\circ p_1^{-1}$, then for $\mu_\phi$--a.e $x \in I$, the conditional measure $\mu_\psi^x$ is exact dimensional on $J_x$, and 
$$
\lim_{r\to 0} \frac{\log \mu_\psi^x(B(y, r))}{\log r}
= \frac{\h_{\mu_\psi}(F)}{\chi_{\mu_\psi}(F)},
$$
for $\mu_\psi^x$--a.e $y \in J_x$; hence, equivalently for $\mu_\psi$-a.e $(x,y) \in J(I)$.

\item[(b)] If $\mu_\phi$ is exact dimensional on $I$, then $\mu_\psi$ is exact dimensional on $I \times Y$. 
\end{itemize}
\end{cor}

\begin{proof}
Item (a) is an immediate consequence of Theorem~\ref{EMRskews}, while part
(b) is an immediate consequence of Theorem~\ref{globalexact}.
\end{proof}

We denote the class of potentials considered in this corollary~\ref{EMRskewsB} by $\mathcal W$.  As an immediate consequence, we get the following.

\begin{cor}\label{f1f2}
Let $f$ be either the  Gauss map $G$ or the induced map $f_{\a, I_0}$ of some Manneville-Pomeau map $f_{\a}$. Consider an open bounded set $Y \subset \R^d$, and  
$
F:I \times Y \lra I\times Y,
$
a H\"older conformal skew product endomorphism over $f$.
Let $\phi:I \lra \R$ be a potential from $\mathcal W$, and set 
$
\psi:=\phi\circ p_1:I\times Y\lra\R.
$
Then, 

\begin{itemize}
\item[(a)] If $\mu_\psi$ is the equilibrium measure of $\psi$, and $\mu_\phi:=\mu_\psi\circ p_1^{-1}$, then for $\mu_\phi$--a.e $x \in I$,  the conditional measure $\mu_\psi^x$ is exact dimensional on $J_x$, and 
$$
\lim_{r\to 0} \frac{\log \mu_\psi^x(B(y, r))}{\log r}
= \frac{\h_{\mu_\psi}(F)}{\chi_{\mu_\psi}(F)},
$$
for $\mu_\psi^x$--a.e $y \in J_x$; hence, equivalently for $\mu_\psi$-a.e $(x,y) \in J(I)$.

\item[(b)] If $\mu_\phi$ is exact dimensional on $I$, then $\mu_\psi$ is exact dimensional on $I \times Y$. 
\end{itemize}
\end{cor}

\sp In \cite{PoW}, \cite{gdms}, \cite{RoUr}, and others, there was developed and studied  the multifractal analysis of equilibrium states for for a class of H\"older continuous summable potentials. As almost always, multifractal analysis requires the use of auxiliary potential and the quantity $T(q)$ commonly referred to as a temperature. We now want to discuss exact dimensionality of equilibrium states of these auxiliary potentials. These fit  to our setting of this section.

If $\phi \in \mathcal W$, then for every $q\ge 1$ and $t>\th:=1/2$ if $f=G$, or $t>\th:=\frac{\a}{\a+1}$ if $f=f_{\a,I_0}$, the potential 
\begin{equation}\label{phiq}
\phi_{q, t}:= -t\log|f'| + q (\phi - \P(\phi)):I\lra\R,
\end{equation}
also belongs to $\mathcal W$. Since 
$
\lim_{t\downto\th}\P\(\phi_{q,t}\)=+\infty
\  \  \  {\rm and} \  \  \
\lim_{t\to+\infty}\P\(\phi_{q,t}\)=-\infty
$,
there exists a unique number $T(q)>\th$ such that
$$
\P(\phi_{q, T(q)}) = 0.
$$
We see that $\P(\phi_{1, 0}) =0$, so $T(1) = 0$. 

Consider a skew product map
$$
F:I\times Y\lra I\times Y
$$ 
over $G$ or over $f_{\a, I_0}$ as in Corollary \ref{f1f2}, and let $\phi \in \mathcal W_0$. If, as usually, $p_1:I \times Y \to I$ is the projection on the first coordinate, define the potentials on $I \times Y$, 
$$
\psi_{q, t} := \phi_{q, t} \circ p_1:I\times Y\lra\R
\  \  \  {\rm and} \  \  \
\psi_q:= \psi_{q, T(q)}:I\times Y\lra\R
$$ 
For every $q\ge 1$, let $\mu_{\psi_q}$, the equilibrium measure of $\psi_q$, defined by formula \eqref{1_2017_01_13}. Let
$$
\mu_q:=\mu_{\psi_q}\circ p_1^{-1}.
$$
The measure $\mu_q$ is called the equilibrium state (measure) for the potential $\phi_{q,T(q)}$ on $I$. We shall prove the following. 

\begin{thm}\label{f1exact}
In the above setting, if $F: I\times Y \to I \times Y$ is either a H\"older conformal skew product endomorphism over the continued fractions transformation $G$ or the induced map $f_{\a, I_0}$, $\a>0$, of the Manneville--Pomeau map $f_\a$, and if $\phi \in \mathcal W$, then for every $q\ge 1$, the measure $\mu_{\psi_q}$ is exact dimensional on $I\times Y$.
\end{thm}

\begin{proof}
From Theorems~1 and 2, and Proposition~3 of \cite{PoW} (see also \cite{gdms} and \cite{OM2}) for much more general treatement), we obtain the exact dimensionality of the measures $\mu_{\phi_q}$ on $I$. Therefore, our theorem follows directly from Corollary~\ref{f1f2}. 

\end{proof}

\section{Diophantine Approximants and the Doeblin-Lenstra Conjecture}\label{doeblin}

We want to apply the results about skew products to certain properties of Diophantine approximants, making the conjecture of Doeblin and Lenstra more general and precise. 
Recall  that the continued fraction (Gauss) transformation is given by:
\beq\label{1_2017_03_03_B}
G(x)=\lt\{\frac{1}{x}\rt\}=\frac1x-n \   \  \text{ if }  \  \  x\in \Big(\frac1{n+1},\frac1n\Big],
\eeq
and $G(x)=0$ otherwise. Then the corresponding coding map
$$
\pi_G:\N^\N\lra [0,1]
$$
is given by the formula
$$
\pi_G(\om)= [\om_1, \om_2, \ldots] = \frac{1}{\om_1+\frac{1}{\om_2+\frac{1}{\ldots + \frac{1}{\om_n+\ldots}}}},
$$
and it is a bijection between $\N^\N$ and the set of irrational numbers in $[0,1]$. 

If we truncate the representation $[\om_1, \om_2, \ldots]$ at an integer $n\ge 1$, then we obtain a rational number $p_n/q_n$, called the $n$-th convergent of $x:=\pi_G(\om)$, where $p_n, q_n \ge 1$ are relatively prime integers, and 
$$
\frac{p_n}{q_n} = [\om_1, \ldots, \om_n]
$$
If needed, we shall also denote $p_n$ and $q_n$ respectively by $p_n(\om)$ and $q_n(\om)$ or also by $p_n(x)$ and $q_n(x)$, in order to indicate their dependence on $\om$ and $x$. We will also sometimes write $\om_n(x)$ for $\om_n$.
Let us now introduce (see for eg \cite{IK}) the numbers
$$
\Theta_n:= \lt|x-\frac{p_n}{q_n}\rt|\cdot q_n^2, \  \ n \ge 1
$$
This number $\Theta_n$ also depends on $\om$ or (equivalently) on $x$ and will be also denoted by $\Theta_n(\om)$ or $\Theta_n(x)$. 
Denote:
$$
T_n=T_n(\om):=\pi_G(\sg^n(\om))= [\om_{n+1}, \om_{n+2}, \ldots], \  \ n \ge 1, \ \text{and}
$$
$$
V_n=V_n(\om):=[\om_n, \ldots, \om_1], \  \  n \ge 1
$$
We will also denote them respectively by $T_n(x)$ and $V_n(x)$.
We see that $T_n(x)$ represents the future of $x$ while the number $V_n(x)$ represents the past of $x$. It follows directly from the definitions that for every integer $n \ge 1$, we have that
\beq\label{1_2017_03_16}
V_n = \frac{q_{n-1}}{q_n}, \  \  \  \Theta_{n-1}=\frac{V_n}{1+T_nV_n},  \  \  \  {\rm and} \  \  \  \Theta_n = \frac{T_n}{1+T_nV_n}
\eeq

We will use again natural extensions systems (see for eg \cite{Ru2}, \cite{M-DCDS}, \cite{M-JSP11}, \cite{M-MZ}) which belong, in certain cases, to our class of skew products. The natural extension $\^G=G:[0,1]\times [0,1]\to [0,1]\times [0,1]$ of the Gauss map $G=G:[0,1]\to (0,1]$ is such an example, and is given by the formula
$$
\^G(x, y)= \lt(T_1(x), \frac{1}{\om_1(x)+y}\rt)
$$
It follows from this that 
$$
\^G(x,0) = \lt(T_1(x), \frac{1}{\om_1(x)}\rt) 
\  \  {\rm and } \  \  \
\^G^2(x,0) = \lt(T_2(x), \frac{1}{\om_2(x)+\frac{1}{\om_1(x)}}\rt)
$$
By induction, we obtain for every $n \ge 1$, that
$$
\^G^n(x,0) = (T_n(x),[\om_n(x), \ldots, \om_1(x)]) = (T_n(x), V_n(x))
$$

The approximation coefficients $\Theta_n$ were the object of an important Conjecture originally stated by Doeblin and reformulated in the 80's by Lenstra (see \cite{IK}), namely that for Lebesgue--a.e $x\in [0, 1)$ the frequency of appearances of $\Theta_n(x)$ in the interval $[0, t]$, $t\in [0,1]$, is given by the function $F:[0,1]\to [0,1]$ given by 
$$
F(t) = 
\begin{cases}\frac{t}{\log 2}   &{\rm if } \  t\in [0,1/2], \\
\frac{1}{\log 2}(1-t + \log 2t)  &{\rm if } \ t\in [1/2,1].
\end{cases}
$$
More precisely, the \textbf{Doeblin-Lenstra Conjecture} says that for Lebesgue--a.e. $x \in [0, 1]$ and all $t \in [0, 1]$, 
$$
\mathop{\lim}\limits_{n \to \infty} \frac{\#\{1 \le k \le n: \Theta_k(x) \le t\}}n=F(t),
$$
i.e. the above limit exists and is equal to equal to the above function $F(t)$. 

This conjecture was solved by Bosma, Jager and Wiedijk in the 1980's, see \cite{BJW}.
In the proof, they needed fundamentally the \textit{natural extension} $[0,1]\times [0,1],\^G,\^\mu_G)$, of the continued fraction dynamical system $\^G$ with the classical Gauss measure $\mu_G$ defined by 
$$
\mu_G(A)=\frac1{\log 2}\int_A\frac1{1+x}\,dx
$$
for every Borel set $A\sbt\R$.
Indeed in the expression of $\Theta_n$ we have both the future $T_n$, as well as the past $V_n$, thus the natural extensions is the right construction in this case.

\sp Let us now apply our results on skew products for the natural extension $\^G$  to the lifts of certain invariant measures, in fact some equilibrium states. 
We recall that the potentials $\phi_s:I\to\R$, given by the formula,  potentials 
$$
\phi_s(x)= -s\log|G'(x)|, \  x\in [0,1). 
$$ 
belong to the class $\cW$ for all $s >1/2$. As in Corollary~\ref{EMRskewsB}, let
$
\psi_s:=\phi_s\circ p_1:[0,1]\times [0,1]\lra\R,
$
i.e. 
$$
\psi_s(x, y) = \phi_s(x), \  \  (x, y) \in [0, 1)^2,
$$
and let $\hat\mu_s=\mu_{\psi_s}$ be the equilibrium measure of $\psi_s$ w.r.t $G$ on $[0, 1)^2$, considered in Corollary~\ref{EMRskewsB} and defined by formula \eqref{1_2017_01_13}. Furthermore, let 
$$
\mu_s:=\mu_{\phi_s}\circ p_1^{-1}
$$ 
be the equilibrium measure of $\phi_s$ considered in the same corollary. From this corollary we know that $\hat\mu_s$ is exact dimensional on $[0, 1) \times [0, 1)$. Our purpose is now to describe the asymptotic frequencies with which $\Theta_n(x)$ come close to arbitrary values, when $x$ is $\mu_s$-\textit{generic}, \ instead of $x$ in a set of full Lebesgue measure as in the original Doeblin-Lenstra conjecture.
We know from this corollary that the measure $\hat\mu_s$ is exact dimensional on $[0,1]\times [0,1]$ and we have a formula for its Hausdorff dimension.

\sp Our purpose now is to describe asymptotic frequencies with which the approximation coefficients $\Th_n(x)$ of $x\in [0,1]$ become close to certain given arbitrary values, when $x$ is Lebesgue non-generic  (i.e $x$ belongs to a set of Lebesgue measure zero). In fact, these $x$'s will be generic for equilibrium measures $\mu_s$, which except for $s=1$ are singular with respect to the Lebesgue measure.

First let us prove the following result about the asymptotic frequency of appearance of $(T_n(x), V_n(x))$ in all squares of $\R^2$, with respect to the measure $\hat \mu_s$:

\begin{thm}\label{distr}
If $s>1/2$, then for $\mu_s$-a.e $x \in [0, 1]$ and for all four real numbers $a<b$ and $c<d$, we have that  
$$
\mathop{\lim}\limits_{n \to \infty} \frac{\#\big\{k\in\{0, 1,\ld, n-1\}: (T_k(x),V_k(x)) \in (a,b) \times (c,d)\big\}}{n} = \hat \mu_s\((a,b) \times (c,d)\)
$$
\end{thm}

\begin{proof}
Denote 
$
A=(a,b) \times (c,d),
$
and for every $\e>0$ let
$$
A(\e):=(a, b) \times (c-\e, d+\e) 
\  \  \  {\rm and }  \  \  \
A(-\e) = (a, b) \times (c+\e, d-\e)
$$
Then
$$
A(-\e) \subset A \subset A(\e)
$$ 
Let $x\in[0,1]\sms \Q$. Then $x = [\om_1(x), \om_2(x), \ldots]$. Hence, there exists an integer $n_\e\ge 1$ such that for every $y \in [0, 1]$ and every integer $n\ge n_\e$, we have that
$$
\big|[\om_n(x), \om_{n-1}(x), \ldots, \om_1(x)+y] - 
[\om_n(x), \om_{n-1}(x), \ldots, \om_1(x)]\big| < \e
$$
Thus, if 
$
\^G^n(x, y)=(T^n(x),[\om_n(x), \om_{n-1}(x), \ldots, \om_1(x)+y]\in A(-\e),
$
then 
$
(T_n(x), V_n(x)) \in A,
$
and 
if $(T_n(x), V_n(x)) \in A$, then 
$
\^G^n(x, y) \in A(\e)
$.
Therefore for every $x\in[0,1]\sms \Q$ and every $y\in[0,1]$, we obtain that
\beq\label{1_2017_03_10} 
\begin{aligned}
\varliminf_{n \to \infty} \frac 1n \sum_{k=0}^{n-1} \1_{A(-\e)}(\^G^k(x, y)) 
&\le \varliminf_{n \to \infty} \frac 1n \sum_{k=0}^{n-1} \1_A(\^G^k(x,0))\le \\ 
&\le \varlimsup_{n \to \infty} \frac 1n \sum_{k=0}^{n-1} \1_A(\^G^k(x,0)) 
\le \varlimsup_{n \to \infty} \frac 1n \sum_{k=0}^{n-1} \1_{A(\e)}(\^G^k(x,y))
\end{aligned}
\eeq
Since the equilibrium measure $\hat\mu_s$ is ergodic on $[0, 1]^2$ with respect to the map $\^G$ and since $\hat\mu_s$ projects on $\mu_s$, the equilibrium state of the potential $\phi_s$, it follows from Birkhoff's Ergodic Theorem that for $\mu_s$-a.e $x \in [0, 1]$ there exist $y_1\in[0,1]$ and $y_2\in[0,1]$ such that 
$$
\lim_{n \to \infty} \frac 1n \sum_{k=0}^{n-1} \1_{A(-\e)}(\^G^k(x,y)) 
=\hat\mu_s(A(-\e)), \ \text{and}
$$
$$
\lim_{n \to \infty} \frac 1n \sum_{k=0}^{n-1} \1_{A(\e)}(\^G^k(x,y)) 
=\hat\mu_s(A(\e))
$$
Along with \eqref{1_2017_03_10} these yield
\begin{equation}\label{hatmus}
\hat\mu_s(A(-\e)) 
\le \varliminf_{n \to \infty} \frac 1n \sum_{k=0}^{n-1} \1_A(\^G^k(x,0)) 
\le \varlimsup_{n \to \infty} \frac 1n \sum_{k=0}^{n-1} \1_A(\^G^k(x,0)) 
\le \hat \mu_s(A(\e))
\end{equation}
Noting that $\hat\mu_s$ does not charge the boundary of $A$ and letting in the above inequality with $\e>0$ to $0$ over a (countable) sequence, we obtain that 
$\mu_s$-a.e $x \in [0, 1]$
$$
\mathop{\lim}\limits_{n \to\infty}\frac 1n \sum_{k=0}^{n-1}\1_A(T_k(x), V_k(x)\) = \hat\mu_s(A)
$$
\end{proof}
 We prove now that for $x \in \Lambda_s$ (recall that $\Lambda_s$ has zero Lebesgue measure, but $\mu_s$-measure equal to 1), the approximation coefficients $\Theta_n(x), \Theta_{n-1}(x)$ behave \textit{very erratically}.
The following Theorem says that for irrational numbers $x \in \Lambda_s$, the behaviour of the consecutive numbers $\Theta_k(x), \Theta_{k-1}(x)$ is chaotic, and we can estimate the asymptotic frequency that $\Theta_k(x)$ is $r$-close to some $z$, while $\Theta_{k-1}(x)$ is $r$-close to some  $z'$. This \textit{asymptotic frequency} is comparable to $r^{\delta(\hat\mu_s)}$, \textit{regardless} of the point $x \in \Lambda_s$ or the numbers $z, z'$ chosen.

\begin{thm}\label{dioph}
For every $s>1/2$, there exists a Borel set $\La_s\sbt[0,1]\sms \Q$ with $\mu_s(\La_s)=1$ and with the following properties:

 1) \ $\HD(\La_s)=\h_{\mu_s}(G)/\chi_{\mu_s}$

2) \ For every $x\in \La_s$ we have that
$
\lim_{n\to\infty}\frac 1n\log \lt|x-\frac{p_n(x)}{q_n(x)}\rt| =\chi_{\mu_s} 
$.

3) \ For every $x\in \La_s$ and $\hat\mu_s$-a.e $(z,z') \in [0, 1)^2$, we have that
$$
\lim_{r\to 0}\log\lim_{n \to \infty} \frac{\#\lt\{0\le k\le n-1:             \big(\Theta_{k}(x),\Theta_{k-1}(x)\big) \in B\lt(\frac{z}{1+zz'}, r\rt) \times B\lt(\frac{z'}{1+zz'},r\rt)\rt\}}{n} 
=\HD(\hat\mu_s)
$$

4)  
$
\ \{\chi_{\mu_s}\}_{s>1/2}=[\chi_{\mu_{1/2}},+\infty).
$

5) $$HD(\hat \mu_s) = \frac{h_{\mu_s}(G)}{\chi_{\mu_s}} + \frac{h_{\hat\mu_s}(\tilde G)}{2\int_{[0, 1)^2} \log(\omega_1(x)+y) \ d\hat\mu_s(x, y)}$$
\end{thm}

\

\begin{proof}
Since $\HD(\mu_s)=\h_{\mu_s}(G)/\chi_{\mu_s}$, there exists a Borel set $\La_s^*\sbt I\sms \mathbb Q$ such that:
 
a) $\mu_s(\La_s^*)=1,
$  

b) $
\HD(\La_s^*)=\h_{\mu_s}(G)/\chi_{\mu_s},
$

c) each Borel subset of $\La_s^*$ with full measure $\mu_s$ has Hausdorff dimension equal to $\h_{\mu_s}(G)/\chi_{\mu_s}$.

Define now $\La_s$ to be the set of all points $x$ in $\La_s^*$ for which c) above holds and for which the assertion of Theorem~\ref{distr} holds. By this theorem, by a result from \cite{PoW}, and by the properties a), b), c), it follows that  the set $\La_s$ satisfies the conditions 1), 2) above, and moreover that  $\mu_s(\La_s) = 1$. 

Now let us show 3). 
So fix $x\in\La_s$ and let arbitrary $z, z' \in[0,1)$. Because of formulas \eqref{1_2017_03_16} there exists some constant $C\ge 1$, such that for all radii $r>0$ we have the following two implications:

1) If for some integer $k\ge 1$, 
$
(T_k(x),V_k(x)) \in B(z, r) \times B(z', r),
$
then 
$
(\Theta_k(x), \Theta_{k-1}(x)) \in B\lt(\frac{z}{1+zz'}, Cr\rt) \times B\lt(\frac{z'}{1+zz'}, Cr\rt).
$

2) 
If 
$
(\Theta_k(x), \Theta_{k-1}(x)) \in B\lt(\frac{z}{1+zz'}, r\rt) \times B\lt(\frac{z'}{1+zz'}, r\rt),
$
then 
$
(T_k, V_k) \in B(z, Cr) \times B(z', Cr).
$

\sp\fr It therefore follows from Theorem~\ref{distr} that 
\beq\label{2_2017_03_16} 
\begin{aligned}
\hat\mu_s\(B(z, C^{-1}&r) \times B(z', C^{-1}r)\) \le \\
&\le \lim_{n \to \infty} \frac{\#\lt\{0\le k\le n-1:             \big(\Theta_{k}(x),\Theta_{k-1}(x)\big) \in B\lt(\frac{z}{1+zz'}, r\rt) \times B\lt(\frac{z'}{1+zz'},r\rt)\rt\}}{n} \le \\ 
& \  \  \  \  \  \  \  \  \  \  \  \  \  \  \  \  \  \  \  \  \  \   \le \hat\mu_s\(B(z, Cr) \times B(z', Cr)\).
\end{aligned}
\eeq
Since by Corollary~\ref{EMRskewsB} the 
measure $\hat\mu_s$ is dimensional exact, we have that 
$$
\lim_{r\to 0}\frac{\log\hat\mu_s\(B(z, C^{-1}r) \times B(z', C^{-1}r)\)}
{\log r}
=\HD(\hat\mu_s)
=\lim_{r\to 0}\frac{\log\hat\mu_s\(B(z, Cr) \times B(z', Cr)\)}
{\log r}.
$$
Along with \eqref{2_2017_03_16} this finishes the proof of formula 3).
By  \cite{PoW}, \cite{gdms}, the function $(1/2,+\infty)\ni s\longmapsto \chi_{\mu_s}$ is strictly increasing and $\lim_{s\to\infty}\chi_{\mu_s}=\infty$. Hence 4) follows. 

Next we compute the Lyapunov exponent of the contraction $y \to \frac{1}{\omega_1(x)+y}$ in the fiber over $x$, which is equal to $2\int_{[0, 1)^2} \log\big(\omega_1(x)+y\big) \ d\hat\mu_s(x, y)$, \ and we have that the entropy of the natural extension is $h_{\hat\mu_s}(\tilde G)$ (see also \cite{M-DCDS}, \cite{M-JSP11}). Hence we use Theorems \ref{codedskews}  and \ref{f1exact} to obtain the Hausdorff dimension of the lift measure $\hat\mu_s$ on $[0, 1) \times [0, 1)$, thus  proving 5). In conclusion all the statements of the Theorem are proved. 

\end{proof}

\textbf{Acknowledgements:} \ The authors thank the referee(s) for several useful suggestions. Eugen Mihailescu thanks the Institut des Hautes \'Etudes Sci\'entifiques, Bures-sur-Yvette, France, for support during a  stay when part of this paper was done.

\

\end{document}